\documentclass[12pt,twoside]{article}
\usepackage{amssymb,amsmath,amsthm}
\usepackage{enumerate}

\usepackage{graphicx}
\usepackage{pictex}
\newcommand{\supp}{\mathop{\rm supp}}

\def\csect#1{Section~\ref{#1}}
\def\capp#1{Appendix~\ref{#1}}
\def\csects#1#2{Sections~\ref{#1} and \ref{#2}}

\let\Prp=\Pr \def\Pr{\Prp\nolimits}
\newcommand{\be}{\begin{equation}}
\newcommand{\ee}{\end{equation}}

\newtheorem{theorem}{Theorem}

\newtheorem{lemma}[theorem]{Lemma}

\theoremstyle{remark}
\newtheorem{question}[theorem]{Question}

\newtheorem{remark}[theorem]{Remark}

\newcommand{\cthm}[1]{Theorem~\ref{#1}}

\newcommand{\clem}[1]{Lemma \ref{#1}}

\newcommand{\cfig}[1]{Figure~\ref{#1}}
\newcommand{\crem}[1]{Remark~\ref{#1}}

\newcommand{\cqes}[1]{Question~\ref{#1}}

\def\ds{\displaystyle}

    \def\bbz{\mathbb{Z}}

    \def\bbr{\mathbb{R}}
    \def\bbn{\mathbb{N}}

\def\M{{\bf M}}\def\L{{\bf L}}
\def\C{{\cal C}}
\def\P{{\cal P}}
\def\Q{{\cal Q}}
\def\muhat{\hat\mu}\def\rhohat{\hat\rho}
\def\murho{\mu^{(\rho)}}\def\mubar{\bar\mu^{(\rho)}}
\def\rhobar{\bar\rho}

\def\IN{I^{(N)}}
\def\AN{A^{(N)}}\def\mN{m^{(N)}}\def\kn{k^{(N)}}
\def\JN{J^{(N)}}\def\jnn{j^{(N)}}
\def\zetan{\zeta^{(N)}}\def\etan{\eta^{(N)}}
\def\PsiN{\Psi^{(N)}}
\def\psiNn{\psi^{(N,n)}}\def\psiNN{\psi^{(N,\mN)}}

\def\interiorM{\rlap{\raise9pt\hbox to9.5pt{\hss$\scriptscriptstyle\circ$}}M}

\def\P{{\cal P}}\def\L{{\cal L}}
\def\M{{\cal M}}\def\T{{\cal T}}\def\To{{\overline{\cal T}}}
\def\t{{\overline t}}
\def\t{{\overline t}}
\def\Xb{\ttttt X}

\def\Xb{\widehat X}

\def\ns{{\#}}

\long\def\kill#1\endkill{\relax}

\def\0{{\it0}}\def\1{{\it1}}\def\2{{\it2}}\def\3{{\it3}}
\def\rng#1#2{\hbox{$(#1\!:\!#2)$}}
\newdimen\mbsize \mbsize=\hsize \multiply\mbsize by 17  \divide\mbsize by 20

\def\Q{{\cal Q}}

\def\P{{\bf P}}\def\Pp{{\bf P}_p}\def\Pmu{{\bf P}_p^\mu}
\def\Pmurho{{\bf P}_p^{\murho}}\def\Prho{{\bf P}_p^{(\rho)}}
\def\Prhop{{\bf P}_{p'}^{(\rho)}}
\def\Pmuhato{\overline{\P}^{\muhat}_p}
\def\Po{\overline{\P}_p}\def\Omegao{\overline{\Omega}}
\def\omegao{\overline{\omega}}
\def\lambdao{\overline{\lambda}_p}
\def\zetao{\bar\zeta}

\def\XN{X^{(N)}}\def\XL{\{0,1\}^{\bbz_L}}\def\VN{V^{(N)}}
\def\muN{\mu^{(N)}}\def\OmegaN{\Omega^{(N)}}

\def\PN{{\bf P}_p^{(N)}}\def\PNn{{\bf P}^{(N,n)}_p}
\def\Pn{{\bf P}^{(n)}_p}
\def\PNnn{{\bf P}^{(N,n+1)}_p}\def\Pk{{\bf P}^{(N,k)}_p}
\def\Pkm{{\bf P}^{(|\pi_m\zeta|,k_m)}_p}

\numberwithin{equation}{section}
\numberwithin{theorem}{section}
\def\rf#1{(P\ref{prop#1})}
\def\ind{{\bf1}}

\begin{document}

\title{\vskip-0.2truein
Stationary States of the One-dimensional Facilitated Asymmetric Exclusion
Process}
\author{A. Ayyer\footnote{Department of Mathematics, 
Indian Institute of Science, Bangalore, 560 012, India.}, 
S. Goldstein\footnote{Department of Mathematics,
Rutgers University, New Brunswick, NJ 08903.},
J. L. Lebowitz\footnotemark[1],
\footnote{Also Department of Physics, Rutgers.}\ \ 
and E. R. Speer\footnotemark[1]}
\date{December 30, 2021}
\maketitle



\begin{abstract}We describe the translation invariant stationary states
(TIS) of the one-dimensional facilitated asymmetric exclusion process in
continuous time, in which a particle at site $i\in\bbz$ jumps to site
$i+1$ (respectively $i-1$) with rate $p$ (resp.~$1-p$), provided that
site $i-1$ (resp.~$i+1$) is occupied and site $i+1$ (resp.~$i-1$) is
empty.  All TIS states with density $\rho\le1/2$ are supported on trapped
configurations in which no two adjacent sites are occupied; we prove that
if in this case the initial state is i.i.d.~Bernoulli then the final state
is independent of $p$.  This independence also holds for the system on a
finite ring.  For $\rho>1/2$ there is only one TIS.  It is the infinite
volume limit of the probability distribution that gives uniform weight to
all configurations in which no two holes are adjacent, and is isomorphic
to the Gibbs measure for hard core particles with nearest neighbor
exclusion.\end{abstract}

\section{Introduction\label{intro}}

The {\it facilitated exclusion process} is a model of particles moving on
a lattice, which we take to be $\bbz^d$.  Our primary interest is in the
one-dimensional version in which particles hop only to nearest neighbor
sites, but for completeness we first describe the general model.  A
configuration of the model is an arrangement of particles on $\bbz^d$,
with each site either empty or occupied by a single particle.  If site
$i$ is occupied, {\it and one of its neighboring sites is also}, then the
particle at site $i$ attempts, at rate 1, to jump to another site $j$,
succeeding only if the target is unoccupied.  The target site $j$ is
chosen with probability $\pi(j-i)$, where $\pi:\bbz^d\to\bbr_+$ is some
probability distribution: $\pi\ge0$ and $\sum_{j\in\bbz^d}\pi(j)=1$.

We will generally consider states of the system---probability measures on
the set of configurations---which have a well-defined density $\rho$, in
the sense that, with probability one, a fraction $\rho$ of the sites in
each configuration are occupied (see \eqref{rhodef} for a precise
statement).  (Here, and throughout unless stated otherwise, by
``measure'' we mean ``probability measure.'')  Since particles are
neither created nor destroyed, the density is a conserved quantity.  If
$\rho$ is not too large there will exist {\it frozen} configurations in
which no two adjacent sites are occupied and hence no particle can move;
the maximum density of such a frozen configuration is clearly 1/2.

Most studies of the model with $d\ge2$ consider the case in which the
target sites are uniformly distributed over the nearest neighbors of the
jumping particle.  For $d=2$, simulations \cite{hl,mcl,rpv} suggest a
somewhat surprising property of the model (which presumably holds for
$d\ge2$): there is a critical density $\rho_c<1/2$ such that, if the
initial state of the models is the measure with Bernoulli i.i.d.~marginals
(which we will refer to simply as {\it Bernoulli measure}) with density
$\rho$, then with probability 1, (i)~for $\rho<\rho_c$ the model
eventually reaches a frozen configuration, while (ii)~for $\rho>\rho_c$
the configuration remains active---that is, particles continue to
jump---for all time.  Note that when $\rho_c<\rho\le1/2$ there exist
frozen configurations with density $\rho$; these are traps for the
dynamics, but with probability 1 they are avoided.  To obtain such a
result rigorously, or indeed any interesting rigorous results, seems very
challenging (but see \cite{ST}).  Indeed, we are not able to prove what
seems to be self evident: that the configurations eventually freeze for
sufficiently small $\rho$, say $\rho<10^{-23}$.

In the remainder of this paper we consider only the case $d=1$, with
probabilities $p$ and $1-p$ of jumps to the right and left, respectively
(that is, we take $\pi(1)=p$, $\pi(-1)=1-p$, and $\pi(j)=0$ for
$j\ne\pm1$).  See \cfig{fig:trans}.  This model is the {\it Facilitated
Asymmetric Simple Exclusion Process} (F-ASEP), with special cases $p=0,1$
({\it Totally Asymmetric}, the F-TASEP) \cite{bbcs,BM,CZ,gkr,GR} and
$p=1/2$ ({\it Symmetric}, the F-SSEP) \cite{BESS,Oliveira}.  A discrete
time version of the F-TASEP was studied in \cite{GLS1,GLS2}.  We write
$X=\{0,1\}^\bbz$ for the configuration space; the condition that
configuration $\eta$ have density $\rho$ is now
 \be\label{rhodef}
\lim_{N\to\infty}\frac1N\sum_{i=1}^N\eta_i
 = \lim_{N\to\infty}\frac1N\sum_{i=-N}^{-1}\eta_i = \rho,
 \ee
and we let $X_\rho$ denote the set of such configurations.  We let
$F\subset X$ denote the set of frozen configurations.

  \begin{figure}[ht!]\begin{center} \label{fig:trans}
\includegraphics[scale=1.5]{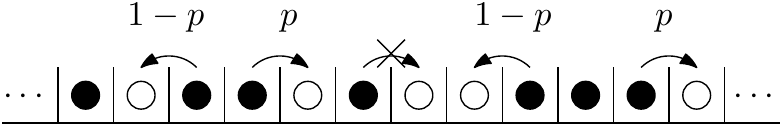}
\caption{Transitions in the one-dimensional F-TASEP}
 \end{center}\end{figure}

We will study the translation invariant (TI) measures on $X$ which are
stationary for the dynamics (TIS measures); for this purpose it suffices
to consider extremal TIS measures, that is, the set of measures such that
every TIS measure is a convex combination of these, and none of these is
a proper convex combination of others.  As we show below, each extremal
measure will be supported on $X_\rho$ for some value of $\rho$.  The
stationary measures for the symmetric case were discussed, for finite
volume, in \cite{Oliveira}; the results given there carry over smoothly
to infinite volume.  The current paper contains two main results, one
each for low density ($0<\rho<1/2)$ and high density ($1/2<\rho<1$),
discussed in \csects{lowrho}{highrho}, respectively.

 For $0<\rho<1/2$, the set of TIS states is simple: all such states are
frozen, that is, are supported on $F$, and all TI measures on $F$ are TIS
measures.  In this case we pose and answer the question: if the initial
state is Bernoulli, what is the final state?  Our main result is that
this final state is independent of the asymmetry parameter $p$.
Moreover, it is also the final state arising from an initial Bernoulli
measure for the discrete-time F-TASEP, in which all particles attempt to
jump at the same times \cite{GLS1,GLS2}.  However, the independence of
the degree of asymmetry does not hold for discrete-time dynamics
\cite{GLS3}.


For $1/2<\rho<1$ we prove that there is a unique TIS state for each value
of $\rho$.  This state may in fact be identified as the Gibbs measure for
a statistical mechanical system in which the only interaction is an
exclusion rule that forbids two adjacent empty sites.  Results
identifying stationary states of particle systems with Gibbs measures
have been established earlier \cite{Georgii,Van}, but under the
assumption that the transition rates for the particle system satisfy a
detailed balance condition, which those of the F-ASEP do not (unless
$p=1/2$).  

Our proof of the uniqueness relies on a coupling with the well-studied
Asymmetric Simple Exclusion Model (ASEP).  The TIS states of density
$\rhohat$ for the latter are precisely the Bernoulli measures
\cite{Liggett} (if $p\ne1/2$ there are also non-TI stationary states,
which are not relevant for our discussion here).  Our coupling yields a
correspondence between these states and the TIS states of the F-ASEP with
density $\rho=1/(2-\rhohat)$.

Our coupling, from the ASEP to the F-ASEP, is a bit more complicated than
the simple map, from the F-ASEP to the ASEP, used in \cite{bbcs}. That map
requires that the particles be labeled, and translation invariant
measures for labeled particle configurations can't be normalized. Thus
the map in \cite{bbcs} is not well suited for our problem of finding the
translation invariant stationary probability measures for the F-ASEP. We
therefore use a coupling that does not require labeling.

\section{The model\label{models}}

In this section we give various definitions and simple results which will
be needed later, and first mention several pieces of general notation.
For any sets $A$ and $B$, function $f:A\to B$, and measure $\lambda$ on
$A$ we let $f_*\lambda$ be the measure on $B$ with
$(f_*\lambda)(C)=\lambda(f^{-1}(C))$; moreover, if $B=\bbr$ we let
$\lambda(f)=\int_Af\,d\lambda$ denote the expected value of $f$ under
$\lambda$.  When $C\subset B$ we let $\ind_C:B\to\{0,1\}$ denote the
indicator function of the set $C$; we will omit mention of the set $B$
when this is clear from the context, as it usually is.  If $S$ is a
finite set then $|S|$ denotes the size of $S$.

We now turn to the models under consideration.  As indicated above, the
configuration space of the F-ASEP model is $X:=\{0,1\}^{\bbz}$.  (In
\csect{fvasep} we will consider also the same dynamics on a ring of $L$
sites with periodic boundary conditions; any notation specific to that
case will be introduced as needed).  Let $F\subset X$ denote the set of
frozen configurations in which no two adjacent sites are occupied, and
similarly let $G\subset X$ denote the set of configurations with no two
adjacent empty sites.  We write $\eta=(\eta(i))_{i\in\bbz}$ for a typical
configuration, and for $j,k\in\bbz$ with $j\le k$ we let
$\eta\rng{j}{k}=(\eta(i))_{j\le i\le k}$ denote the portion of the
configuration lying between sites $j$ and $k$ (inclusive).  We will
occasionally use string notation for configurations or partial
configurations, writing for example
$\eta\rng04=\eta(0)\cdots\eta(4)=10100=(10)^20$.

We denote by $\tau$ the translation operator: if $\eta\in X$ then
$(\tau\eta)(i)=\eta(i-1)$, if $f$ is any function on $X$ then
$\tau f(\eta)=f(\tau^{-1}\eta)$, and if $\mu$ is a (Borel) measure on $X$
then $\tau$ acts on $\mu$ as $\mu\mapsto\tau_*\mu$.  We let $\L(X)$
denote the space of functions $f:X\to\bbr$ for which $f(\eta)$ depends on
the values $\eta(k)$ for only finitely many sites $k$, $\C(X)$ denote the
space of real-valued continuous functions on $X$ (for the topology
generated by $\L(X)$), and $\M=\M(X)$ denote the space of translation
invariant probability measures on $X$ (equipped with the Borel
$\sigma$-algebra, i.e., the natural product $\sigma$-algebra, on $X$).
  
We now turn to a formal specification of the system.  The dynamics is
controlled by {\it Site Associated Poisson Processes} (SAPPs); two of
these, controlling rightward and leftward jumps, respectively, are
associated with each site $i\in\bbz$.  Specifically, given a
TI measure $\mu\in\M$ which specifies the initial distribution of the
system, we consider the probability space $(\Omega,\Pmu)$:
 \be\begin{aligned}\label{OmegaP}
\Omega&=X\times\Omega_0,\quad\text{with}\quad
 \Omega_0=\prod_{i\in\bbz}\bigl(\T^{(i,r)}\times\T^{(i,l)}\bigr),\\
  \Pmu&=\mu\times\Pp,\quad\text{with}\quad
  \Pp=\prod_{i\in\bbz}\bigl(\lambda_p^{(i,r)}\times\lambda_p^{(i,l)}\bigr).
 \end{aligned}\ee
 Here, for $i\in\bbz$ and $\#=l$ or $r$,
 \be\label{SAPP}
 \T^{(i,\#)}=\Bigl\{\bigl((i,t^{(i,\#)}_j)\bigr)_{j=1,2,\ldots}\,\Big|\,
\text{$0<t^{(i,\#)}_1<t^{(i,\#)}_2\cdots$,
$\lim_{j\to\infty}t^{(i,\#)}_j=\infty$}\Bigr\},
 \ee
 and under $\lambda_p^{(i,r)}$ (respectively $\lambda_p^{(i,l)}$) the points
of $\T^{(i,r)}$ (respectively $\T^{(i,l)})$ are distributed as a Poisson
process of rate $p$ (respectively $1-p$.)


The configuration now evolves as follows: at each time $t=t^{(i,r)}_j$ a
particle jumps from site $i$ to site $i+1$ if
$\eta_{t-}(i-1)=\eta_{t-}(i)=1-\eta_{t-}(i+1)=1$, and at each time
$t=t^{(i,l)}_j$ a particle jumps from $i$ to $i-1$ if
$\eta_{t-}(i+1)=\eta_{t-}(i)=1-\eta_{t-}(i-1)=1$.  This so-called {\it
Harris graphical construction} (or {\it Poisson construction} in
\cite{Swart}) leads \cite{Swart} to a process $\eta_t$, well-defined on
$\Omega$, with generator $L$ which acts on $\L(X)$ via
 \be\label{generator}\begin{aligned}
 Lf(\eta)&=\sum_{i\in\bbz}c(i,\eta)[f(\eta^{i,i+1})-f(\eta)].
 \end{aligned}\ee
 Here $\eta^{i,j}$ denotes the configuration $\eta$ with the values of
$\eta(i)$ and $\eta(j)$ exchanged.  The rates $c(i,\eta)$ are given by
 \be\label{rates}
c(i,\eta)= \begin{cases}
   p,&\hbox{if $\eta(i-1)=\eta(i)=1$ and $\eta(i+1)=0$,}\\
  1-p,&\hbox{if $\eta(i)=0$ and $\eta(i+1)=\eta(i+2)=1$,}\\
 0,&\hbox{otherwise.}
 \end{cases}\ee
 $L$ is the generator of a Markov semigroup $S(t)=e^{Lt}$ on $\C(X)$,
which thus acts on $\M(X)$ via $\mu\to\mu_t=\mu S(t)$, where
$(\mu S(t))(f)=\mu(S(t)f)$ or equivalently
$(\mu S(t))(A)=\int_XQ_t(\eta,A)\,d\mu$, with
$Q_t(\eta,A)=(S(t){\bf1}_A)(\eta)$ the transition kernel of the Markov
process.  We will assume that this process, and others to be considered
later, have right-continuous sample paths.

\begin{remark}\label{construction}Since the set of all Poisson times for
different sites will a.s.~be dense in $(0,\infty)$, one cannot perform
all the particle jumps in temporal order, and some care is needed to show
that the construction is well defined.  Details are given in
\cite[Sections 4.3 and 4.4]{Swart}.  Later we carry out such a
construction with a different but equivalent definition of the dynamics,
in which the Poisson times at which particles can jump are associated
with the particles rather than with the sites ({\it Particle Associated
Point Processes}).  See the proof of \cthm{coupling}.\end{remark}

If $\mu$ is a TI measure on $X$  then, by the ergodic theorem,
$\mu$-almost every configuration $\eta$ has a density, i.e., 
 \be\label{density}
r(\eta)=\lim_{N\to\infty}\frac1{2N+1}\sum_{i=-N}^N\eta(i)
 \ee
 exists almost surely.  \eqref{density} defines a map $r:X\to[0,1]$; we
will say that a TI measure $\mu$ {\it has density $\rho$} if
$r(\eta)=\rho$ $\mu$-a.s.  Note that if $\mu$ has density $\rho$ then
$\mu(\eta(i))=\rho$ for any $i$, but that the former is a stronger
statement, ruling out, for example, the possibility that $\mu$ is a
superposition of measures of different densities.  The next lemma shows
that in seeking to describe the set of all stationary TI measures $\mu$
it suffices to consider those for which $\mu(F)$ is 0 or 1 and for which
$r(\eta)$ is $\mu$-a.s.~constant.

\begin{lemma}\label{densities} Every TIS measure $\mu$ on $X$ is a convex
combination of TIS measures for which $r$ is a.s.~constant and $F$ has
measure 0 or 1.  \end{lemma}

\begin{proof} Let $\nu=r_*\mu$; $\nu$ is a measure on $[0,1]$ which gives
the distribution of the density under $\mu$.  Then (see for example
\cite{Kallenberg}) there exists a regular conditional probability
distribution for $\mu$, that is, a family $\{\mu_\rho\mid \rho\in[0,1]\}$
of probability measures on $X$ such that $\mu_\rho$ has density $\rho$
and for any measurable $A\subset X$,
 \be
\mu(A)=\int_{0\le \rho\le 1}\mu_\rho(A)d\nu(\rho).
 \ee
 Moreover, $\{\mu_\rho\}$ is unique in the sense that for any other such
family $\{\mu'_\rho\}$, $\mu_\rho=\mu'_\rho$ $\nu$-a.s. If we further
write $\mu_\rho=\mu_\rho\big|_F+\mu_\rho\big|_{X\setminus F}$ we obtain,
after normalization, the desired representation.  It remains to verify
that these normalized measures,
$\mu_\rho\big|_F\big/\bigl(\mu_\rho\big|_F(X)\bigr)$ and
$\mu_\rho\big|_{X\setminus
F}\big/\bigl(\mu_\rho\big|_{X\setminus F}(X)\bigr)$, are TIS measures.

Now
 \be
\mu(A)=\tau_*\mu(A)=\int_{0\le \rho\le 1}\tau_*\mu_\rho(A)d\nu(\rho),
 \ee
  and from \eqref{density} it follows that $\tau_*\mu_\rho$ has density
$\rho$, so that the uniqueness of the conditional probability
distribution implies that $\tau_*\mu_\rho=\mu_\rho$ $\nu$-a.s.
Stationarity of $\mu_\rho$ $\nu$-a.s.~follows similarly from the fact
that neither dynamics destroys or creates particles.  Finally,
translation invariance and stationarity of $\mu_\rho\big|_F$ and
$\mu_\rho\big|_{X\setminus F}$ follows from the fact that $F$ is
translation invariant and invariant under the dynamics.  \end{proof}

The key idea in the next lemma appears in \cite{Oliveira} in the context
of a system on a ring.

\begin{lemma}\label{exclusive} If $\mu$ is a TIS measure on $X$ then
$\mu(F\cup G)=1$.  \end{lemma}

\begin{proof} The argument we give requires that $p$ be strictly
positive; by the symmetry of the model under simultaneous spatial
reflection and the replacement $p\to1-p$, we may assume that this
condition holds.  Suppose that $\mu(F\cup G)<1$.  If $\eta\notin F\cup G$
then $\eta$ contains two adjacent zeros and two adjacent ones; let $2k$
be the minimum number of sites by which a double zero follows a double
one---that is, for which the string $11(01)^k00$ occurs in a
configuration---with nonzero probability.  We first note that $k\ge1$
almost surely. For from \eqref{rates}, the  translation invariance of
$\mu$, and $\mu=\mu_t$ for all $t$,
 \begin{align}\nonumber
\frac{d}{dt}\mu\bigl(\eta(0{:}1)=11\bigr)
  &=-p\,\mu\bigl(\eta(0{:}2)=110\bigr)-(1-p)\,\mu\bigl(\eta(-1{:}1)=011\bigr)\\
  &\hskip30pt +p\,\mu\bigl(\eta(-2{:}1)=1101\bigr)
     +(1-p)\,\mu\bigl(\eta(0{:}3)=1011\bigr)\nonumber\\
  &=-p\,\mu\bigl(\eta(0{:}3)=1100\bigr)
      -(1-p)\,\mu\bigl(\eta(-2{:}1)=0011\bigr).\nonumber
 \end{align}
  This quantity must vanish, since $\mu$ is stationary, and since $p$ is
nonzero the probability of $1100$ occurring is zero.

 Now by the choice of $k$, $\mu\bigl(\eta(0{:}2k+3)=11(01)^k00\bigr)>0$.
Then a simple calculation as above, using repeatedly the fact that for
any $i$ and for $j<k$, $\mu\bigl(\eta(i{:}i+2j+3)=11(01)^j00\bigr)=0$,
shows that
 \be
 \frac{d}{dt}\mu\bigl(\eta(2{:}2k+3)=11(01)^{(k-1)}00\bigr)
   = p\,\mu\bigl(\eta(0{:}2k+3)=11(01)^k00\bigr)>0,
 \ee
  contradicting stationarity.\end{proof}

\begin{remark}\label{simple_configs}Let $\eta^*\in X$ be the period-two
configuration defined by $\eta^*(i)=i\bmod2$.  $\eta^*$ and its translate
$\tau\eta^*$ consist of alternating 1's and 0's, and the measure
$\mu^*=(\delta_{\eta^*}+\delta_{\tau\eta^*})/2$ is a TIS measure with
density $\rho=1/2$.  Note that $\mu^*(F)=\mu^*(G)=1$. \end{remark}

\begin{theorem}\label{contsum} Let $\mu$ be a TIS measure on $X$ with
density $\rho$.  Then:
 \par\smallskip\noindent
 (a) If $\rho<1/2$ then  $\mu(F)=1$, i.e.,  $\mu$ is supported on $F$.
 \par \smallskip\noindent
 (b) If $\rho=1/2$ then $\mu=\mu^*$ (see \crem{simple_configs}).
 \par\smallskip\noindent
 (c) If $\rho>1/2$ then $\mu(G)=1$, i.e.,  $\mu$ is supported on $G$.
\end{theorem}

\begin{proof} We know that $\mu$ is supported on $F\cup G$.  Suppose that
$\eta\in\supp \mu$; then we may assume that $\eta\in F\cup G$ and
$r(\eta)=\rho$.  If $\rho=r(\eta)<1/2$ then, by \eqref{density}, $\eta$
must contain a positive density of double zeros and so lie in
$F\setminus G$, verifying (a); similarly, if $\rho>1/2$ then
$\eta\in G\setminus F$, verifying (c).  If $\rho=1/2$ then \eqref{density} with
$\eta\in F\cup G$ implies that $\eta$ does not have a positive density of
either double ones or double zeros, and hence almost surely has no double
ones or double zeros at all, verifying (b).  \end{proof}

\section{The low density region\label{lowrho}}

In this section we consider TIS states on $X_\rho$ with $0<\rho<1/2$; by
\cthm{contsum} these are necessarily supported on $F$.  In fact, any TI
measure supported on $F$ is clearly a TIS state; in \crem{FTI} we obtain
a prescription for obtaining all such states from a construction
introduced there for the study of the high density region.  Here we
address the following question:

\begin{question}\label{keyq}If the system is given an initial measure
$\murho$, the Bernoulli measure with density $0<\rho<1/2$, what is the
final measure?  \end{question}
 
\subsection{The totally asymmetric model\label{tasep}}

In this section we address \cqes{keyq} for the totally asymmetric model
(F-TASEP); we take $p=1$ in \eqref{rates} but the discussion for $p=0$
would be similar.  The answer is given in \cite{GLS1,GLS2} for the
discrete-time F-TASEP, and the analysis there applies almost unchanged in
the continuous-time case, so we content ourselves with a brief summary.

First, it is convenient to enlarge the state space of the process from
$X$ to $\Xb:=X\times\bbz$, writing the state of the system at time $t$ as
$(\eta_t,J_t)$.  In this new version of the model, $J_t$ is the signed
count of the number of particles passing between sites 0 and 1 up to time
$t$.  The new version is defined on the same probability space
$(\Omega,\Pmu)$ as the original one (see \eqref{OmegaP}, \eqref{SAPP}),
with $J_t$ incremented or decremented by 1  at, and only at, those times
$t^{(0,r)}_j$ or $t^{(1,l)}_j$, respectively, at which jumps actually
occur; it is straightforward to verify, as in \cite{Swart}, that this
leads to a well-defined process.  We always assume that $J_0=0$.

The variable $J_t$ allows us to introduce the {\it height profile}
$h_t:\bbz\to\bbz$ associated with the pair $(\eta_t,J_t)\in\Xb$ (see,
e.g., \cite{ISS}), defined by the requirements that
$h_t(i)-h_t(i-1)=1-2\eta_t(i)=(-1)^{\eta_t(i)}$ for all $i\in\bbz$ and
$h_t(0)=2J_t$, or more explicitly by
 \be\label{profdef} h_t(i)=\begin{cases}
  2J_t+\ds\sum_{j=1}^i(-1)^{\eta_t(j)},& \hbox{if $i\ge0$,}\\ 
  \noalign{\vskip3pt}
   2J_t-\ds\sum_{j=i+1}^{0}(-1)^{\eta_t(j)},& \hbox{if $i<0$.}\end{cases}
 \ee
  Then, since $0<\rho<1/2$, $\lim_{i\to\pm\infty}h_t(i)=\pm\infty$.
Moreover, as a function of $t$, $h_t$ is monotonically increasing; in
particular, $h_t(i)$ increases by 2 when a particle jumps from site $i$
to site $i+1$, and such an increase can occur only if $h_t(i-1)>h_t(i)$.
See the inset in \cfig{fig:profiles}.

We can now sketch the determination of the fate of an arbitrary initial
configuration $\eta_0\in X_\rho$; full details are given in \cite{GLS2}.
Define $Q=Q(\eta_t)\subset\bbz$ by
$Q:=\{q\in\bbz\mid h_t(q)> \sup_{i<q}h_t(i)\}$.  From the observation
above on how $h_t$ can increase it follows that $Q(\eta_t)$ is in fact
independent of $t$ and that for $q\in Q$, $h_t(q)$ is constant.  If $q$
and $q'$ are consecutive elements of $Q$ and $i\in\bbz$ satisfies
$q\le i<q'$, then $h_t(i)<h_t(q')=h_0(q')$, so that
$\lim_{t\to\infty}h_t(i)$, and hence also
$\eta_\infty(i)=\lim_{t\to\infty}\eta_t(i)$, exist.  Further, since
$i+h_0(i)$ is even for all $i$ and $h_0(q')-h_0(q)=1$, $q'-q$ is odd,
and so necessarily
 \be\label{lowetainf}
\eta_\infty\rng{q+1}{q'}=1\,0\,1\,0\,\cdots\,1\,0\,0=(1\,0)^{(q'-q-1)/2}0. 
 \ee
 See \cfig{fig:profiles}.  This completes the determination of
$\eta_\infty$.

\begin{figure}[ht!]  \begin{center}
\includegraphics[width=5truein,height=2.5truein]{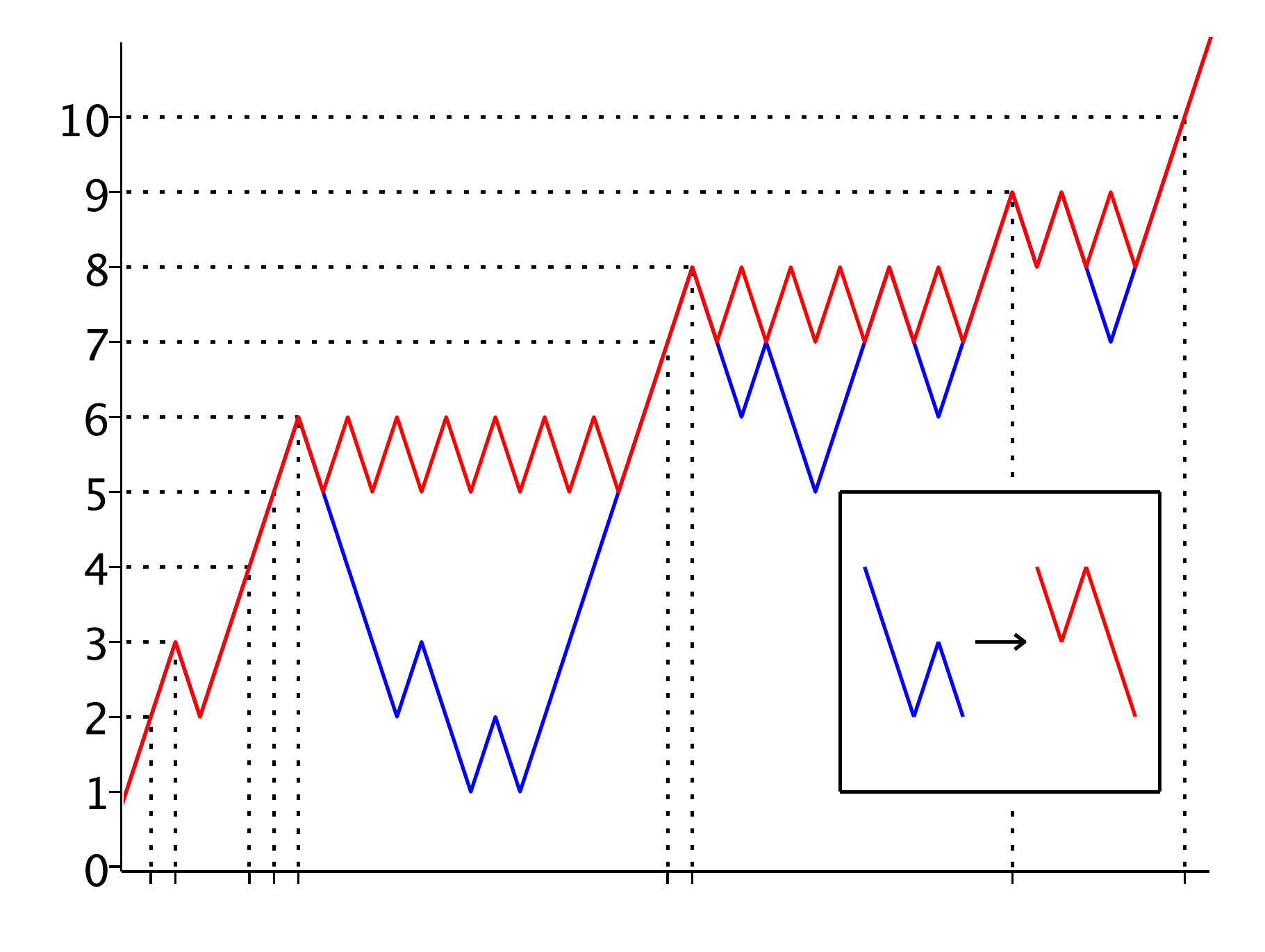}
\caption{Portion of typical initial (blue, lower) and final (red, upper)
height profiles in the F-TASEP; the profile increases monotonically in
time from the former to the latter.  The vertical dotted lines are at
sites in $Q$.  The inset shows the local change in the profile associated
with the jump of a single particle, carrying local configuration
$1\,1\,0\,1$ to $1\,0\,1\,1$.}\label{fig:profiles} \end{center}
\end{figure}

 We now suppose that $\eta_0$ is distributed according to the Bernoulli
distribution $\murho$, and write $\Prho$ rather than $\Pmurho$.  We ask
for the distribution $\murho_\infty=\eta_{\infty*}\Prho$ of
$\eta_\infty$, where here we think of $\eta_\infty$ as a map
$\eta_\infty:\Omega\to X_\rho$.  Let $V=\{\eta\in X\mid 0\in Q(\eta)\}$;
note that \eqref{lowetainf} implies that $V$ coincides, up to a set of
$\murho_\infty$-measure zero, with $\{\eta\mid\eta\rng{-1}{0}=0\,0\}$, and
with $\murho_\infty\bigl(\{\eta\mid\eta\rng{-1}{0}=1\,1\}\bigr)=0$ this
implies that $\murho_\infty(V)=1-2\rho$.  To obtain $\murho_\infty$ we
will first describe the conditional measure $\murho_\infty(\cdot\mid V)$,
then obtain $\murho_\infty$ as the (unique) TI measure with this
conditional measure.  For $\eta_0\in V$ we may index $Q=Q(\eta_0)$ as
$Q=\{q_k\}_{k\in\bbz}$, taking $q_0=0$ and requiring that the $q_k$ be
increasing in $k$; then we may specify $\murho_\infty(\cdot\mid V)$ by
giving the joint distribution of the variables $n_k$ defined by
$q_{k+1}-q_k=2n_k+1$.

It is easy to see that the $n_k$ are i.i.d.  To describe the distribution
of a single $n_k$ we recall the {\it Catalan numbers} \cite{Stanley}
 \be\label{catalan}
   c_n=\frac1{n+1}\binom{2n}n,\quad n=0,1,2,\ldots\,;
 \ee
  $c_n$ counts the number of strings of $n$ 0's and $n$ 1's in which the
number of 0's in any initial segment does not exceed the number of 1's.
If $q=q_k\in Q$ and $l=q+2n+1$ then $q_{k+1}=l$ if and only if
$h_0(l)=h_0(q)+1$ and $h_0(i)\le h_0(q)$ for $q<i<l$, and there are $c_n$
strings $\eta\rng{q+1}{l-1}$ satisfying this condition and hence yielding
$q_{k+1}=l$.  Since each such string has probability $\rho^n(1-\rho)^{n+1}$ we
have sketched a proof of the next theorem, which is taken from
\cite{GLS2}.  Recall that $\tau$ denotes translation and $\ind_C$ the
indicator function of $C$.

\begin{theorem}\label{ftmeas}(a) \label{ldrenewal}The random variables
$n_k$ of the F-TASEP, defined on $V$ as above, are i.i.d.~under
$\murho_\infty(\cdot\mid V)$, with distribution
 \be\label{catdist} \murho_\infty(\{n_k=n\}\mid V)
    =c_n\rho^n(1-\rho)^{n+1}, \quad n=0,1,2,\ldots.
 \ee

 \smallskip\noindent
 (b) The measure $\murho_\infty$ is given by 
 \be
   \murho_\infty=\frac1Z\sum_{m\ge0}\sum_{i=0}^{2m}\tau_*^{-i}
 \bigl(\ind_{V_m}\murho_\infty(\cdot\mid V)\bigr)
=\sum_{m\ge0}\sum_{i=0}^{2m}\tau_*^{-i}
 \bigl(\ind_{V_m}\murho_\infty\bigr),
 \ee
 where $V_m=\{\eta_0\in V\mid q_1-q_0=2m+1\}$ and
 $Z=\murho_\infty(V)^{-1}$ is a normalizing constant.
\end{theorem}

\subsection{The  model in finite volume\label{fvasep}}

In this section we address \cqes{keyq}, or rather an appropriately
modified version of it, for the F-ASEP on a periodic ring of $L$ sites.
The system size $L$ will be constant during our analysis and we typically
suppress $L$-dependence. We first discuss the totally asymmetric model
and describe the result corresponding to \cthm{ftmeas}, then show that
the limiting measure is in fact independent of the asymmetry parameter
$p$.  The ring is denoted $\bbz_L:=\{0,1,\ldots,L-1\}$; we consider a
system of $N$ particles on these sites, governed by the obvious
modification of the F-ASEP dynamics defined in \csect{models}.  For a
configuration $\eta\in \XL$ we let $|\eta|:=\sum_{i=1}^L\eta(i)$ denote
the number of particles in $\eta$; the configuration space of our model
is then $\XN:=\{\eta\in\XL\mid|\eta|=N\}$.  We will be interested in the
fate of an initial measure $\muN$ which is uniform on $\XN$; the
probability space is then $(\OmegaN,\PN)$:
 \be\begin{aligned}\label{OmegaL}
\OmegaN&=\XN\times\prod_{i\in\bbz_L}\bigl(\T^{(i,r)}\times\T^{(i,l)}\bigr),\\
  \PN&=\muN\times\prod_{i\in\bbz_L}
   \bigl(\lambda_p^{(i,r)}\times\lambda_p^{(i,l)}\bigr).
 \end{aligned}\ee
  Here $\T$ and $\lambda$ are as in \eqref{OmegaP}.  (In view of our
earlier use of $\murho$ and $\Prho$, writing $\muN$ and $\PN$ is
admittedly an abuse of notation, but we believe that this will not give
rise to confusion.)   The construction of
the dynamics is parallel to the construction in infinite volume, but is
technically simpler because the considerations of \crem{construction} do
not apply; we omit details.  The auxiliary variable $J_t$ is not needed
here.

Now consider the F-TASEP, taking $p=1$ above.  Given an initial
configuration $\eta_0\in \XL$, with $|\eta_0|<L/2$, we extend $\eta_0$ to
an $L$-periodic configuration $\eta_0^*$ on $\bbz$, apply the
construction of \csect{tasep} to obtain $Q(\eta_0^*)$, and let
$Q(\eta_0):=Q(\eta^*_0)\cap\{0,1,\ldots,L-1\}$; $Q(\eta_0)$ will contain
$L-2|\eta_0|$ sites.  An argument as in infinite volume shows that the
limiting configuration $\eta_\infty$ exists and satisfies
\eqref{lowetainf} for $q,q'$ consecutive (in cyclic order) elements of
$Q(\eta_0)$ (with the expression $q'-q-1$ in the exponent of
\eqref{lowetainf} interpreted mod $L$).

Now fix $N<L/2$; we will determine the distribution
$\muN_\infty=\eta_{\infty*}\PN$ of $\eta_\infty$ when $\eta_0$ is
distributed according to $\muN$ (this is the modified version of
\cqes{keyq} referred to above).  Let
$\VN:=\{\eta\in\XN\mid 0\in Q(\eta)\}$ and note that
$|\VN|=(L-2N)\binom LN/L$, since if one partitions $\XN$ into equivalence
classes under translation then each class contains a fraction $(L-2N)/L$
of elements belonging to $\VN$.  We can determine the conditional measure
$\muN_\infty(\cdot\mid\VN)$ by simple counting: given
$0=q_0<q_1<\ldots<q_{L-2N-1}\le L-1$, with
$q_{i+1}-q_i\equiv2n_i+1\pmod{L}$ for $0\le n_i<L/2$, there are
$\prod_{i=0}^{L-2N-1}c_{n_i}$ initial configurations $\eta_0\in\XN$ with
$Q(\eta_0)=\{q_0,\ldots,q_{L-2N-1}\}$, all leading to
$\eta_\infty=\eta^{(q_0,\ldots,q_{L-2N-1})}$, where
 \be
\eta^{(q_0,\ldots,q_{L-2N-1})}
   := (10)^{q_1-q_0}0(10)^{q_2-q_1}0\cdots0(10)^{q_0-q_{L-2N-1}+L}0.
 \ee
  Thus we have

\begin{theorem}\label{finmeas}
(a) The possible limiting configurations of the F-TASEP model on $\VN$
  are the $\eta^{(q_0,\ldots,q_{L-2N-1})}$, and
 \be
 \muN_\infty\bigl(\{\eta^{(q_0,\ldots,q_{L-2N-1})}\}\mid\VN)
  = \frac{\displaystyle{L\prod\nolimits_{i=0}^{L-2N-1}c_{n_i}}}
         {(L-2N)\binom LN}.
 \ee

 \par\smallskip\noindent
(b) $\muN_\infty=\displaystyle{
   \frac1L\sum_{i=0}^{L-1}\tau^i_*\muN_\infty(\cdot\mid\VN)}$.
\end{theorem}


We now consider the general F-ASEP model on $\bbz_L$, with partially
asymmetric dynamics governed by the asymmetry parameter $p$.  Since from
any initial configuration $\eta_0$ there is a sequence of possible
transitions leading to a frozen configuration,
$\eta_\infty=\lim_{t\to\infty}\eta_t$ exists almost surely, for any
$\eta_0$, and is frozen.  The distribution
$\muN_\infty=\eta_{\infty*}\PN$ of the limiting configurations is then
well defined; our goal is to show that this distribution is independent
of $p$ (as our notation indicates).  The next lemma is simple---it
follows immediately from elementary consideration of the dynamics---but
will be useful in several places.

\begin{lemma}\label{doublez}Suppose that $\eta_t$ is the state at time
$t$ of a process evolving via the F-ASEP dynamics, for a system either
of $N$ particles on $L$ sites or in infinite volume.  If for some site
$i$ and time $t$, $\eta_t(i)=\eta_t(i+1)=0$, then also
$\eta_s(i)=\eta_s(i+1)=0$ for all $s<t$, respectively $\PN$ or $\Prho$
almost surely.\end{lemma}

\clem{doublez} implies that if two adjacent sites are empty in
$\eta_\infty$ then they must also be empty in all $\eta_t$, $t\ge0$.
Because of this it is convenient to decompose configurations into {\it
components}---strings of 1's and 0's within which no two adjacent sites
are empty but which are separated from each other by (at least) two
adjacent empty sites.  See \cfig{fig:pimaps}.  (Formally a {\it
component} of a configuration $\eta$ is the restriction $\eta|_I$ of
$\eta$ to an interval $I=\{i,i+1,\ldots,j\}$ for which
$\eta(i)=\eta(j)=1$, $\eta(i-2)=\eta(i-1)=\eta(j+1)=\eta(j+2)=0$, and
there is no site $k$ in $I$ such that $\eta(k)=\eta(k+1)=0$.)  We let
$c(\eta)$ denote the number of components in $\eta$, and write $\PNn$ for
the measure $\PN$ conditioned on the event $c(\eta_0)=n$.

\begin{theorem}\label{main} For all $L$ and $N$, with $N<L/2$, the
measure $\muN_\infty$ is independent of $p$, and so is given by
\cthm{finmeas}(b).  \end{theorem}

\begin{proof}We will prove by induction on $n$, $n=1,2,\ldots$, that for
all $L$ and $N$ with $L/2>N\ge n$ the distribution of $\eta_\infty$ under
$\PNn$ is independent of $p$.  The theorem then follows from
$\PN(\cdot)=\sum_{n=1}^N\PNn(\cdot)\PN(c(\eta_0)=n)$, since the
distribution $\muN$ of $\eta_0$ is independent of $p$.  The case $n=1$ of
the induction is trivial: if the initial configuration has a single
component then so does the final one, and for any $p$ this component is
just $101\cdots01$ (with $N$ 1's) and its position will be uniformly
distributed over the ring, by translation invariance.

We now assume inductively that $n$ is such that for all $k\le n$ and all
$L,N$ with $L/2>N\ge n$, the distribution of $\eta_\infty$ under $\Pk$ is
independent of $p$.  We then fix a configuration $\zeta\in\XN$ and show
that $\PNnn(\eta_\infty=\zeta)$ is independent of $p$; we may assume that
no two consecutive sites are occupied in $\zeta$, since otherwise this
probability is 0.  Consider first the case $c(\zeta)>1$; then if
necessary we may rotate $\zeta$ (which does not affect the conclusion) to
an orientation in which there exists a site $i$, with $3\le i\le L-3$,
such that $\zeta(0)=\zeta(1)=0$, $\zeta(i)=\zeta(i+1)=0$, and
$\zeta(j)=1$ for at least one $j$ in the set $\{2,\ldots,i-1\}$ and at
least one $j$ in $\{i+2,\ldots,L-1\}$.  Given $i$, we define maps
$\pi_1,\pi_2:\XN\to\bigcup_{0\le N'\le N}X^{(N')}$ by
$\pi_1\eta=\eta\ind_{\{2,\ldots,i-1\}}$ and
$\pi_2\eta=\eta\ind_{\{i+2,\ldots,L-1\}}$, i.e.,
 \begin{align}
(\pi_1\eta)(j)&=\begin{cases}
  \eta(j),& \text{if $2\le j\le i-1$,}\\
  0,&\text{if $0\le j\le 1$ or $i\le j\le L-1$},\end{cases}\\
(\pi_2\eta)(j)&=\begin{cases}
  \eta(j),& \text{if $i+2\le j\le L-1$,}\\
   0,&\text{if $0\le j\le i+1$.}\end{cases}
 \end{align}
  See \cfig{fig:pimaps}. 

\begin{figure}[ht!]  \begin{center}
\includegraphics[width=4.5truein]{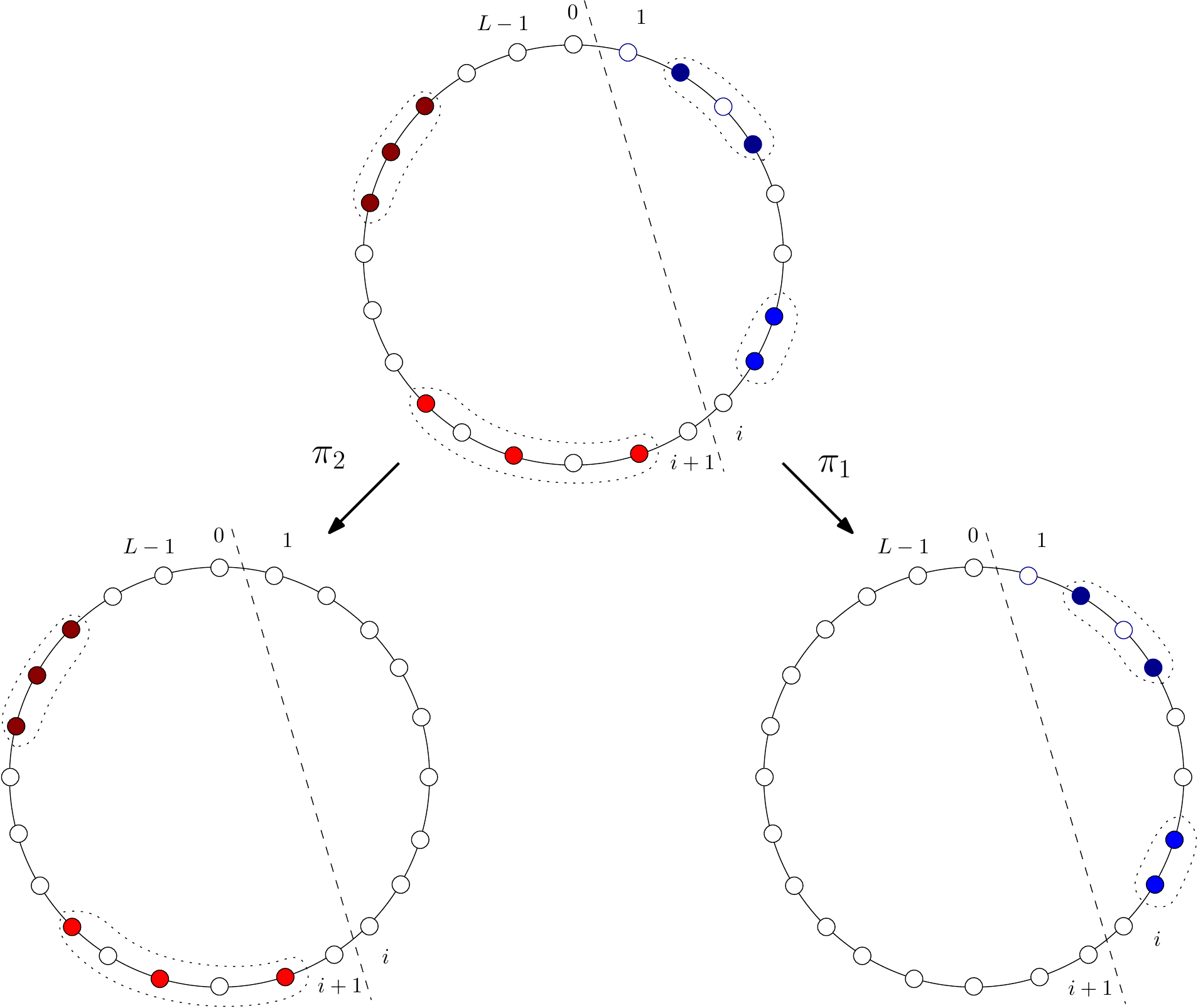} \caption{A typical
configuration in $E_{2,2}$ (see below) of the finite-volume
F-ASEP, with its image under the maps $\pi_1$ and $\pi_2$.  The four
components are enclosed in dotted loops and are also indicated by color
coding.} \label{fig:pimaps} \end{center} \end{figure}

 For integers $k_1,k_2$ with
$1\le k_m\le|\pi_m\zeta|$, satisfying $k_1+k_2=n+1$, define the events
 \begin{align*}
 E_{k_1,k_2}&:=\{\eta_0\in\XN\mid\eta_0(0)=\eta_0(1)=\eta_0(i)=\eta_0(i+1)=0,\\
 &\hskip75pt|\pi_m\eta_0|=|\pi_m\zeta|,\text{ and }
    c(\pi_m\eta_0)=k_m,\ m=1,2\},\\
\noalign{\noindent and for $m=1,2$,}
  E^{(m)}_{k_m}&:=\{\eta_0\in X^{(|\pi_m\zeta|)}\mid\eta_0=\pi_m\eta_0
  \text{ and }     c(\eta_0)=k_m\}.
 \end{align*}
 We will assume in what follows that $k_1,k_2$ are chosen so that
$\PNnn(\eta_\infty=\zeta\mid E_{k_1,k_2})\ne0$.  Since
 \be
\PNnn(\eta_\infty=\zeta)
   =\sum_{k_1,k_2}\PNnn(\eta_\infty=\zeta\mid E_{k_1,k_2})\PNnn(E_{k_1,k_2})
 \ee
  and $E_{k_1,k_2}$ depends only on the initial configuration $\eta_0$,
it suffices to show that for all such $k_1,k_2$,
$\PNnn(\eta_\infty=\zeta\mid E_{k_1,k_2})$ is independent of $p$.

Now observe that
 \be\label{ind}
 \PNnn(\eta_\infty=\zeta\mid E_{k_1,k_2})
  =\prod_{m=1,2}\Pkm(\eta_\infty=\pi_m\zeta\mid E^{(m)}_{k_m}).
 \ee
 This is because (i)~whatever the value of $i$ used to define $\pi_1$ and
$\pi_2$, $\pi_1\eta_0$ and $\pi_2\eta_0$ are independent (and hence
remain so after conditioning on $E_{k_1,k_2}$), and (ii)~since for all
histories $\eta_t$ arriving at $\eta_\infty=\zeta$, \clem{doublez}
implies that the sites $0$, $1$, $i$, and $i+1$ must always be empty,
the probability on the left, given $\eta_0$, is determined by the SAPP
times $t_j^{(i',\#)}$ for $2\le i'\le i-1$ and $i+2\le i'\le L-1$, and
these are independent.  But also we have
 \be\label{cond}
 \Pkm(\eta_\infty=\pi_m\zeta\mid E^{(m)}_{k_m})
 = \frac{\Pkm(\eta_\infty=\pi_m\zeta)}{\Pkm(E^{(m)}_{k_m})},
 \ee
 since, with probability 1 with respect to
$\Pkm$, $\eta_\infty=\pi_m\zeta$ is possible only if $E^{(m)}_{k_m}$
occurs.  Both the numerator and denominator on the right hand side of
\eqref{cond} are independent of $p$, the numerator by the inductive
assumption and the denominator since it involves only the initial
condition.  Thus $\Pkm(\eta_\infty=\pi_m\zeta\mid E^{(m)}_{k_m})$ and
hence, by \eqref{ind}, $ \PNnn(\eta_\infty=\zeta\mid E_{k_1,k_2})$, are
independent of $p$.

 This completes the proof of the $p$-independence of
$\PNnn(\eta_\infty=\zeta)$ in the case $c(\zeta)>1$.  From the validity
of this result for all such $\zeta$ it follows that
 \be
 \PNnn(c(\eta_\infty)=1)
  =1-\sum_{l=2}^N\sum_{\{\zeta\mid c(\zeta)=l\}}\PNnn(\eta_\infty=\zeta)
 \ee
 is also independent of $p$.  Then by  translation invariance,
$\PNnn(\eta_\infty=\zeta)=\PNnn(c(\eta_\infty)=1)/L$ if $c(\zeta)=1$.
This completes the proof.  \end{proof}

\subsection{The partially asymmetric model in infinite volume\label{ivpasep}}

In this section we return to the ($p$-dependent) F-ASEP dynamics on
$\bbz$; we assume that $\eta_0$ is distributed as the Bernoulli measure
$\murho$, $0<\rho<1/2$, and show that the distribution of the limiting
configuration $\eta_\infty$ is independent of $p$.  As in \csect{tasep}
we write $\Prho$ for the measure $\Pmurho$ on the sample space $\Omega$
of \eqref{OmegaP}.  We begin with two preliminary results; the first is
standard.

\begin{lemma}\label{mixing}If $f:\Omega\to X$ commutes with
  translations then $f_*\Prho$ is mixing under translations.
\end{lemma}

\begin{proof}Note that $\Omega$ is in fact a product space,
$\Omega=\prod_{i\in\bbz}\bigl(\{0,1\}\times\T^{(i,r)}\times\T^{(i,l)}\bigr)$,
and that $\Prho$ is a product measure.  Thus $\Prho$ is certainly mixing
under translations.  Then for any measurable sets $A,B\subset X$,
 \begin{align*}
\lim_{n\to\infty}f_*\Prho(A\cap\tau^nB)
&=\lim_{n\to\infty}\Prho(f^{-1}(A\cap\tau^nB))\\
&=\lim_{n\to\infty}\Prho(f^{-1}(A)\cap\tau^nf^{-1}(B))\\
&=\Prho(f^{-1}(A))\Prho(f^{-1}(B))\\
&=   f_*\Prho(A)f_*\Prho(B).\qedhere
 \end{align*}
\end{proof}
 
\begin{lemma}\label{limit2}$\eta_\infty:=\lim_{t\to\infty}\eta_t$ exists,
and is frozen, $\Prho$-almost surely.
\end{lemma}

\begin{proof}The cases $p=1$ and $p=0$ follow from the discussion of
\csect{tasep}, so we may suppose that $0<p<1$.  We will show that for any
$L>0$ there exist, $\Prho$-almost surely, two pairs of adjacent sites,
one on each side of the interval $[-L,L]$, which are empty for all times.
The interval between these pairs of sites is isolated from any outside
influence; it is effectively a finite system in which any initial
configuration can, and therefore almost surely will, eventually freeze.  

We now fill in the details of the argument.  For $t\in\bbz_+$ define
$\theta_t:\Omega\to X$ by $\theta_t(k)\bigl(=\theta_t(\omega)(k)\bigr)=1$
if $\eta_t(k)=\eta_t(k+1)=0$, $\theta_t(k)=0$ otherwise.  \clem{doublez}
implies that the sequence $(\theta_t)_{t\in\bbz_+}$ is pointwise
decreasing and so $\theta_\infty=\lim_{t\to\infty}\theta_t$ exists;
$(\theta_\infty)_*\Prho$ is mixing by \clem{mixing} and moreover, since
$\Prho\bigr(\theta_t(k)=1\bigr)\ge1-2\rho$ for all $t$,
$\Prho(\theta_\infty(k)=1)\ge1-2\rho$ by the Monotone Convergence
Theorem.  Thus if for $k,l>0$ we define
$\Omega_{k,l}=\{\omega\in\Omega
\mid \theta_\infty(-k-2)=\theta_\infty(l+1)=1\}$ then for any $L>0$,
$\Prho$-a.e. $\omega\in\Omega$ will lie in $\Omega_{k,l}$ for some
$k,l>L$.

We claim that, $\Prho$-almost surely on $\Omega_{k,l}$,
$\lim_{t\to\infty}\eta_t\big|_{[-k,l]}$ exists and is frozen; by the
previous paragraph this suffices for the result.  Now conditioning on
$\Omega_{k,l}$ simply implies, for the behavior of $\eta_t$ in $[-k,l]$,
that $\eta_0\big|_{[-k,l]}$ has at most $\lfloor(k+l+1)/2\rfloor$ particles
and that no transitions occur across the bonds $\langle-k-1,-k\rangle$
and $\langle l,l+1\rangle$, and with these restrictions there is, from
any initial configuration, a sequence of possible transitions leading to
a frozen configuration.\end{proof}

Now set $\mu_\infty:=\eta_{\infty*}\Prho$; $\mu_\infty$ is mixing by
\clem{mixing}.  Our main result, \cthm{nop} below, is that $\mu_\infty$
does not depend on $p$.

\begin{theorem}\label{nop} For all $\rho$, with $0<\rho<1/2$, the
distribution of $\eta_\infty$ is independent of $p$, and so is given by
\cthm{ftmeas}(b).
\end{theorem}

\begin{proof} Let $I\subset\bbz$ be an interval of integers, let
$\zeta\in\{0,1\}^I$ be a configuration on $I$, and let $E$ be the event
that $\eta_\infty\big|_I=\zeta$. We will show that $\Prho(E)$ is
independent of $p$, proving the result.  We may assume that $\zeta$
contains no pair of adjacent occupied sites, since otherwise
$\Prho(E)=0$.

Choose $l\in\bbn$ so large that
$I\subset [-l,l]$.  Then, since $\mu_\infty$ is mixing and hence ergodic,
and in $\mu_\infty$ there is a strictly positive density $1-2\rho$ of
pairs of adjacent empty sites, there must $\Prho$-almost surely exist sites
$i$ and $j$, with $i<-l$ and $j>l$, such that
$\eta_\infty(i-1)=\eta_\infty(i)=\eta_\infty(j)=\eta_\infty(j+1)=0$.
Focusing on the maximal such $i$ and minimal such $j$ leads to the
representation
 \be\label{union}
 E=\bigcup_{i<-l<l<j} F_i\cap E\cap F'_j,
 \ee
 where for some $m,m'\ge0$,
$F_i:=\{\eta_\infty(i-1{:}-l-1)=00(10)^m\hbox{ or } 00(10)^m1\}$ and
$F'_j:=\{\eta_\infty(l+1{:}j+1)=(01)^{m'}00\hbox{ or }1(01)^{m'}00\}$.  Since
\eqref{union} is a disjoint union,
 \be\label{shel}\begin{aligned}
\Prho(E)
 &=\sum_{i<-l<l<j}\Prho(F_i\cap E\cap F'_j)\\
 &= \sum_{i<-l<l<j}
    \alpha(i,j,l)\Prho(G_i\cap G'_j).
\end{aligned}\ee
 where $G_i:=\{\eta_\infty(i)=\eta_\infty(i-1)=0\}$,
$G'_j:=\{\eta_\infty(j)=\eta_\infty(j+1)=0\}$, and
$\alpha(i,j,l):=\Prho(F_i\cap E\cap F'_j\mid G_i\cap G'_j)$. 

 We show below, and assume for the moment, that as the notation
indicates, $\alpha(i,j,l)$ is independent of $p$.  Now fix $\epsilon>0$;
since $\Prho(G_i)=\Prho(G'_j)=1-2\rho$, \clem{mixing} implies that there
is an $l^*_p$ such that
 \be\label{miximp}
   |\Prho(G_i\cap G'_j)-(1-2\rho)^2|<\epsilon \text{ for $l\ge l^*_p$.}
 \ee
Now \eqref{shel} and \eqref{miximp} imply that if $\epsilon<(1-2\rho)^2$,
 \be
\sum_{i<-l<l<j}\alpha(i,j,l)
     <\frac1{(1-2\rho)^2-\epsilon}.
 \ee
 But then for any $p,p'$ with $0\le p,p'\le1$ we have for
 $l>\max(l^*_p,l^*_{p'})$, 
 \be
 |\Prho(E)-\Prhop(E)|<2\epsilon\sum_{i<-l<l<j}\alpha(i,j,l)
  \le\frac{2\epsilon}{(1-2\rho)^2-\epsilon}.
 \ee
 Since $\epsilon$ is arbitrary, $\Prho(E)=\Prhop(E)$.
 
To show that $\alpha(i,j,l)=\Prho(F_i\cap E\cap F'_j\mid G_i\cap G'_j)$ is
independent of $p$ we appeal to \cthm{main}.  Let $L=j+1-i$, let $N$
denote the number of particles in $\eta_0$ which lie in the interval
$[i-1,j+1]$, and let $H_n=\{N=n\}$.  If $G_i$ and $G'_j$ occur then
necessarily $N<L/2$, and we may write
 \be\label{decomp} \begin{aligned}
 \Prho(F_i\cap E&\cap F'_j\mid G_i\cap G'_j)\\
 &= \sum_{n<L/2} \Prho(F_i\cap E\cap F'_j\mid G_i\cap G'_j\cap H_n)
\Prho(H_n).
 \end{aligned}\ee
  Now consider a system with $n$ particles on a ring of $L$ sites, which
for convenience we label as $i,i+1,\ldots j$; the state space is
$X^{(n)}\subset\{0,1\}^{[i,j]}$ and there is a natural map
$\chi_n:H_n\to X^{(n)}$ given by restriction.  Further, 
 \be\label{equal}
   \Prho(F_i\cap E\cap F'_j\mid G_i\cap G'_j\cap H_n)
 =  \Pn(\chi_n(F_i\cap E\cap F'_j)\mid\chi_n(G_i\cap G'_j)),
 \ee
 since, under the conditioning on $G_i\cap G'_j$ and $\chi_n(G_i\cap G'_j)$,
respectively, if some initial condition and sequence of particle jumps
in the system on $\bbz$ produces $G_\theta\cap E\cap G'_\sigma$ then the
corresponding initial condition and sequence of jumps in the finite system
will produce $\chi_n(G_\theta\cap E\cap G'_\sigma)$.  Since the right hand
side of \eqref{equal} is independent of $p$ by \cthm{main}, so is
$\Prho(F_i\cap E\cap F'_j\mid G_i\cap G'_j)$, by \eqref{decomp}.
\end{proof}
\noindent

\section{The high density region\label{highrho}}

We now turn to consideration of the TIS measures for the F-ASEP with
density $\rho>1/2$.  By \cthm{contsum} such measures are supported on the
set $G\subset X$ of configurations with no two adjacent holes, so in this
section we will regard the F-ASEP as a Markov process on $G$, and write
$\M(G)$ for the space of TI probability measures on $G$.  We will prove:

\begin{theorem}\label{gibbs}For each $\rho>1/2$ there is a unique TIS
measure with density $\rho$ for the F-ASEP. \end{theorem}

\noindent
This result was established for the symmetric ($p=1/2$) model in
\cite{BES}.  

Some context for the result arises from a familiar equilibrium
statistical mechanical system of particles on a one-dimensional lattice,
sometimes referred to as the {\it nearest-neighbor hard core} model, in
which the only interaction is an infinitely strong repulsion between
particles on adjacent sites, so that the possible configurations are
those with no two particles adjacent.  When this system is considered on
a ring all configurations satisfying this restriction are equally likely,
and in the thermodynamic limit there is a unique (for given density
$\rho$) Gibbs measure.  If we exchange the roles of particles and holes
we obtain from this a measure $\mubar$ supported on $G$, and this measure
is a TIS state for the F-ASEP, whatever the asymmetry---which must, of
course, be the unique such state identified in \cthm{gibbs}.

Results identifying all stationary states of particle systems as
canonical Gibbs measures have been established in fairly general contexts
in \cite{Georgii} and \cite{Van} (see also \cite{Spohn} for a review of
the situation).  These results, however, require that the rates satisfy a
detailed balance condition, which ours do not unless $p=1/2$, and even in
this symmetric case certain non-degeneracy hypotheses on the rates
exclude the F-SSEP.  Rather than attempting to extend or modify the
arguments of these papers we give an independent proof of the theorem for
the F-ASEP, based on a coupling with the Asymmetric Simple Exclusion
Model (ASEP).  

Recall \cite[Chap. VIII]{Liggett} that the ASEP has configuration space
$Y\,(=X)= \{0,1\}^\bbz$; we will write a typical configuration in $Y$ as
$\zeta=\bigl(\zeta(i)\bigr)_{i\in\bbz}$ and write $\M(Y)$ for the space
of TI probability measures on $Y$.  The ASEP dynamics is defined in
parallel with that of the F-ASEP (see \csect{models}), using the same
SAPPs $\bigl((i,t^{(i,r)}_j)\bigr)_{j=1,2,\ldots}$ and
$\bigl((i,t^{(i,l)}_j)\bigr)_{j=1,2,\ldots}$ and requiring that a
particle jump from $i$ to $i+1$ at $t=t^{(i,r)}_j$ if
$\eta_{t-}(i)=1-\eta_{t-}(i+1)=1$ and from $i$ to $i-1$ at
$t=t^{(i,l)}_j$ if $\eta_{t-}(i)=1-\eta_{t-}(i-1)=1$.  See \cite{Liggett}
for more details, including a specification of the generator $\hat L$;
we will use below the evolution operator $\hat S(t)=e^{\hat Lt}$, whose
action on measures is defined in parallel to that of $S(t)$, defined in
\csect{models} .  It is known \cite[Theorem~VIII.3.9(a)]{Liggett} that
for $\0\le\rhohat\le1$ the Bernoulli measure is the unique TIS state of
density $\rhohat$ for the ASEP.

To define the coupling we first introduce the map $\phi:Y\to G$ defined
by the substitutions $1\to1$, $0\to10$; more specifically, for
$\zeta\in Y$,
 \be\label{subs}
\phi(\zeta)
  =\cdots\psi(\zeta(-1))\psi(\zeta(0))\psi(\zeta(1))\psi(\zeta(2))\cdots,
 \ee
 where $\psi(1)=1$, $\psi(0)=10$, and the substitution is made so that
$\psi(\zeta(1))$ begins at site 1.  Further, we define
$\gamma_\zeta:\bbz\to\bbz$ so that $\gamma_\zeta(i)$ is the initial site
of the string $\psi(\zeta(i))$ which is substituted for $\zeta(i)$ under
$\phi$: $\gamma_\zeta(i)=2i-\sum_{j=1}^i\zeta(j)$ if $i\ge1$, 
$\gamma_\zeta(i)=2i-1+\sum_{j=i}^0\zeta(j)$ if $i\le0$.  For example, if
$\zeta\rng{-1}{4}=0\,1\,1\,0\,0\,1$ then
$\phi(\zeta)\rng{-2}{6}=1\,0\,1\,1\,1\,0\,1\,0\,1$ and
$\gamma_\zeta(-1)=-2$, $\gamma_\zeta(0)=0$, $\gamma_\zeta(1)=1$,
$\gamma_\zeta(2)=2$, $\gamma_\zeta(3)=4$, and $\gamma_\zeta(4)=6$; see
\cfig{fig:gamma}.  $\phi$ is clearly a bijection of $Y$ with $G_1$, where
$G_\sigma$ denotes the set of configurations $\eta\in G$ with
$\eta(1)=\sigma$; we write $\phi^{-1}:G_1\to Y$ for the inverse of this
bijection.

\begin{figure}[ht!]  \begin{center}
\includegraphics[width=5truein]{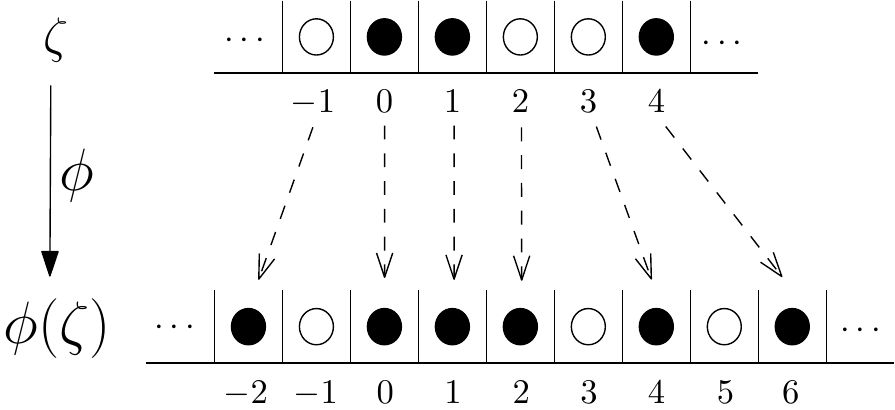}

 \caption{Portion of a typical configuration $\zeta$, with
$\zeta\rng{-1}{4}=0\,1\,1\,0\,0\,1$, and its image under $\phi$.  The
dashed arrows show how $\gamma_\zeta$ maps the sites for $Y$:
$\gamma_\zeta(-1)=-2$, etc.} \label{fig:gamma} \end{center} \end{figure}

Suppose now that $\muhat\in\M(Y)$ and that $\muhat$ has density
$\rhohat$.  $\phi_*\muhat$ cannot be TI, since it is supported on
$G_1$.  However, $\phi$ does give rise to a map $\Phi:\M(Y)\to\M(G)$,
obtained as follows.  Write $G_1=G_{10}\cup G_{11}$, where
$G_{\sigma\sigma'}:=\{\eta\in G\mid\eta(1)=\sigma,\eta(2)=\sigma'\}$.
$G_0=G\setminus G_1$ is just the translate $\tau^{-1}G_{10}$, and for
$\muhat\in\M(Y)$, $\Phi(\muhat)$ on $G_0$ should be just the translate of
$\phi_*(\ind_{G_{10}}\muhat)$.  This leads us to define
 \be\label{defPhi}
\Phi(\muhat)
 :=\rho\Bigl(\phi_*\muhat + \tau_*^{-1}\phi_*(\ind_{G_{10}}\muhat)\Bigr).
 \ee
 $\Phi(\muhat)$ is then easily seen to be TI.  The normalizing constant
$\rho$ has value $\rho=1/(2-\rhohat)$, since $\muhat(Y)=1$ and
$\muhat(G_{10})=1-\rho$; $\rho$ is also the density $\Phi(\muhat)(G_1)$
of $\Phi(\muhat)$, since $\phi_*\muhat(G_1)=1$ and
$G_{10}\cap\tau G_1=\emptyset$.  $\Phi:\M(Y)\to\M(G)$ is a bijection with
inverse
$\Phi^{-1}(\mu)=\mu(G_1)^{-1}\phi^{-1}_*\bigl(\mu\big|_{G_1}\bigr)$.
Moreover, $\Phi$ preserves convex combinations and this, with the
invertibility of $\Phi$, implies that $\muhat$ is ergodic (i.e.,
extremal) if and only if $\Phi(\muhat)$ is.  Finally, as we shall see in
\cthm{TIS} below, $\Phi$ also commutes with the time evolutions in the
two systems.

\begin{remark}\label{FTI} As noted in Section 3, the TIS states of the
F-ASEP at low density, $0<\rho<1/2$, are precisely the TI measures
supported on $F$.  Such measures may be obtained by defining a map from
$X$ to $F$ via the substitutions $0\to0$, $1\to01$, in parallel with
\eqref{subs}; then as with \eqref{defPhi} we obtain a bijective
correspondence between the set of all TI measures on $X$ of density
$\rhobar$ and the set of TI measures on $F$ of density
$\rhobar/(1+\rhobar)$. \end{remark}

 For $\zeta\in Y$ we let $K_\zeta=\zeta^{-1}(1)$ be the set of particle
locations in $\zeta$, and let $(k_i)_{i\in\bbz}$ be an ordered
enumeration of $K_\zeta$ ($k_i<k_{i'}$ if $i<i'$).  If $\eta\in G$ is an F-ASEP configuration then
we will refer to certain particles in $\eta$ as {\it true} particles; the
true particles are those which are immediately followed by another
particle.  The mapping $\gamma_\zeta$, when restricted to $K_\zeta$, then
gives a bijective correspondence between the particles in the ASEP
configuration $\zeta$ and the true particles in $\phi(\zeta)$.  Note that
if $\eta\in G$ is any F-ASEP configuration and there exists an ordered
enumeration $(k'_i)_{i\in\bbz}$ of the sites of the true particles in
$\eta$ satisfying
 \be\label{trans}
k'_{i+1}-k'_i=\gamma_\zeta(k_{i+1})-\gamma_\zeta(k_i)\,
  \bigl(=2(k_{i+1}-k_i)-1\bigr)
 \ee
 for all $i$,  then $\eta$ is a translate of $\phi(\zeta)$.

The idea behind the coupling is to establish the correspondence between
particles in the ASEP and true particles in the F-ASEP at time 0, through
$\gamma_{\zeta_0}$, and then to maintain this correspondence as the
configurations evolve.  As a preliminary we introduce a minor
modification of the F-ASEP dynamics: we keep the SAPPs
$\bigl((i,t^{(i,r)}_j)\bigr)_{j=1,2,\ldots}$ and
$\bigl((i,t^{(i,l)}_j)\bigr)_{j=1,2,\ldots}$ introduced in \csect{models}
through \eqref{SAPP}, but replace the exchanges which they trigger by
exchanges corresponding to those in the ASEP.  Thus at a time
$t=t^{(i,r)}_j$ an exchange occurs only if
$\eta_{t-}\rng{i}{i+2}=1\,1\,0$, and then the true particle at $i$
exchanges with the pair $1\,0$ to its right, yielding
$\eta_t\rng{i}{i+2}=1\,0\,1$.  Similarly for $t=t^{(i,l)}_j$: if
$\eta_{t-}\rng{i-2}{i+1}=1\,0\,1\,1$ then the true particle at $i$
exchanges with the pair to its left, yielding
$\eta_t\rng{i-2}{i+1}=1\,1\,0\,1$.  It is clear that these are the same
exchanges which took place in the earlier formulation of the dynamics,
although triggered by Poisson times associated with different sites, so
that the process defined in this way is the same as the F-ASEP process
defined earlier.  A formal proof of this is easily given.

\begin{theorem}\label{coupling} For any $\zeta_0\in Y$ there exists a
process $(\zeta_t,\eta_t)$, with state space $Y\times G$, such that
$\zeta_t$ is the ASEP process with initial configuration $\zeta_0$ and
$\eta_t$ is the F-ASEP process with initial configuration
$\eta_0=\phi(\zeta_0)$.  Moreover, for all $t$, $\eta_t$ is a translation
of $\phi(\zeta_t)$.  \end{theorem}

\begin{proof} To avoid consideration of irrelevant special cases we will
assume that $K:=K_{\zeta_0}$, the set of initial ASEP particle positions,
is infinite to both the left and right of the origin; this case suffices
for the application we make of the theorem in \cthm{TIS} below (and in
fact other cases are simpler to treat).  We regard $K$ as a set of labels
for the ASEP particles, and keep these labels as the particles move to
different sites.  $K$ also labels the true particles in the F-ASEP
through the map $\gamma_{\zeta_0}$.  We will obtain the coupled dynamics
from a set of {\it Particle Associated Poisson Processes} (PAPPs),
defined on the probability space $(\Omegao,\Pmuhato)$ (compare
\eqref{OmegaP}--\eqref{SAPP}):
 \be\begin{aligned}  \label{OmegaPo}
\Omegao&=Y\times\Omegao_0,\quad\text{with}\quad
 \Omegao_0=\prod_{k\in K}\bigl(\To^{(k,r)}\times\To^{(k,l)}\bigr),\\
 \Pmuhato&=\muhat\times\Po,\quad\text{with}\quad
  \Po=\prod_{k\in K}\bigl(\lambdao^{(k,r)}\times\lambdao^{(k,l)}\bigr),\\
 \To^{(k,\#)}&=\Bigl\{\bigl((k,\t^{(k,\#)}_j)\bigr)_{j=1,2,\ldots}\,\Big|\,
\text{$0<\t^{(k,\#)}_1<\t^{(k,\#)}_2\cdots$,
$\lim_{j\to\infty}\t^{(k,\#)}_j=\infty$}\Bigr\},
 \end{aligned}\ee
 with $\lambdao^{k,r}$ and $\lambdao^{k,l}$ Poisson processes with rates
$p$ and $1-p$, respectively.  We take the initial measure $\muhat$ to be
$\delta_{\zeta_0}$ and write $\Po^{(\zeta_0)}$ rather than
$\Po^{\delta_{\zeta_0}}$.

Given a specific realization of the PAPPs we define the corresponding
space-time configuration
$(\zeta,\eta)=\bigl((\zeta_t(i),\eta_t(i))\bigr)_{t\in[0,\infty),i\in\bbz}$
as follows (see \crem{construction}).  For each $N\in\bbn$ we let
$\IN:=[\kn_1,\kn_2]$ be the minimal interval for which
(i)~$\kn_1,\kn_2\in K$, (ii)~$\kn_1\le0<\kn_2$, and (iii)~no Poisson events
occur for particles $\kn_1$ and $\kn_2$ during the time interval $[0,N]$.
Formally, the particles $\kn_1$ and $\kn_2$ are inactive during the time
interval $[0,N]$ and thus insulate the sites in $\IN$ from outside
influence during this time interval.  Note that $\IN$ exists a.s.~and
that clearly $\IN\subset I^{(N+1)}$ and $\IN\nearrow\bbz$ a.s.  

The next step is to define the space-time configuration $(\zetan,\etan)$
on $[0,N]\times\IN$:
$(\zetan,\etan)=\bigl((\zetan_t(i),\etan_t(i))\bigr)_{t\in[0,N],i\in\IN}$.
The definition is such that $\zetan_0=\zeta_0\big|_{\IN}$ and that the
restriction of $(\zeta^{(N+1)},\eta^{(N+1)})$ to $[0,N]\times\IN$ is
$\zetan$.  Once this is done we define
$(\zeta,\eta)=\lim_{N\to\infty}(\zetan,\etan)$, where of course the limit
exists trivially.

To define $\zetan$ we first specify that particles with labels $\kn_1$ and
$\kn_2$ (in either the ASEP or F-ASEP) stay at their initial position
through time $N$: $\zetan_t(\kn_1)=\zetan_t(\kn_2)=1$ and
$\eta_t(\gamma_{\zeta_0}(\kn_1))=\eta_t(\gamma_{\zeta_0}(\kn_2))=1$ for
$0\le t\le N$.  Next, note that the set of Poisson events
$(k,\t^{(k.\ns)}_j)$ with $\kn_1<k<\kn_2$, $\t^{(k,\ns)}_j \in[0,N]$, and
$\ns=r$ or $l$, is a.s.~finite.  Taking these events in their time
order, we specify that the particles move as follows:

\begin{itemize}

\item for an event $(k,\t^{(k.r)}_j)$: if at time $\t^{(k.r)}_j\!\!\!\!-$
the site to the right of the ASEP particle with label $k$ is empty, then
that particle moves to its right, and the F-ASEP particle with label $k$
exchanges with the $1\,0$ pair on its right (as in the modified F-ASEP
dynamics above), and

\item for an event $(k,\t^{(k.l)}_j)$: if at time $\t^{(k.l)}_j\!\!\!\!-$
the site to the left of the ASEP particle with label $k$ is empty, then
that ASEP particle moves to its left and the F-ASEP particle with label
$k$ exchanges with the $1\,0$ pair to its left.

\end{itemize}

It is clear intuitively that the first and second components of this
process are respectively the ASEP and F-ASEP as defined earlier using the
SAPPs; we give a formal proof of this in \capp{same}.  Moreover, one
checks easily that if $K_{\zeta_t}=(k_i)_{i\in\bbz}$ and the set of true
particles in $\eta_t$ is $(k'_i)_{i\in\bbz}$ then \eqref{trans} is
satisfied, so that $\eta_t$ is a translate of $\phi(\zeta_t)$.
\end{proof}

For the next main result we need a lemma.  Let $\hat S(t)$ and $S(t)$ be
the evolution operators for the ASEP and F-ASEP, defined respectively
earlier in this section and in \csect{models}.

\begin{lemma}\label{preserve}$S(t)$ and $\hat S(t)$ preserve ergodicity.
\end{lemma}

\begin{proof}We prove the lemma for $S(t)$; the proof for $\hat S(t)$ is
the same.  The lemma follows from two elementary observations: (i) a
covariant image of an ergodic measure is ergodic (for our purposes here,
a covariant image of a measure $P$ is a measure $f_*P$, where $f$
commutes with translations); and (ii)~the product of an ergodic dynamical
system with one that is weakly mixing is ergodic.  (i) is trivial; (ii)
is a well-known fact that the reader can easily verify (or find in
\cite{Hochman}).

The lemma then follows from the observation that for any measure $\mu$ on
$X$ we have that $\mu S(t)=\eta_{t*}\Pmu$ (see \eqref{OmegaP}).  (Here it
is irrelevant whether $\eta_t$ is defined on $\Omega$ using the original
jump rule or the modified one.)  Since $\Pmu=\mu\times\Pp$ with the
product measure $\Pp$ mixing and hence weakly mixing, (i) and (ii) imply
that $\mu S(t)$ is ergodic if $\mu$ is.  \end{proof}

 The idea of the proof of the following result is taken from \cite{GLS2}:

\begin{theorem}\label{TIS}(a) For any TI measure $\muhat$ on $Y$,
$\Phi(\muhat)S(t)=\Phi\bigl(\muhat\hat S(t)\bigr).$

 \smallskip\noindent
 (b) $\Phi$ is a bijection of the TIS measures for the ASEP and F-ASEP
systems.  \end{theorem}

\begin{proof} (b) is an immediate consequence of (a) and the remark
directly below \eqref{defPhi} that $\Phi$ is a bijection of TI measures,
and clearly it suffices to verify (a) for $\muhat$ ergodic.  Let us write
$\nu_t:=\Phi(\muhat)S(t)$ and
$\widetilde\nu_t:=\Phi\bigl(\muhat\hat S(t)\bigr)$.  Since $S(t)$ and
$\hat S(t)$ preserve ergodicity, as does $\Phi$, $\nu_t$ and
$\widetilde\nu_t$ are ergodic, so that these two measures are either
equal or mutually singular.  Hence to prove their equality it suffices to
find a nonzero measure $\lambda_t$ with $\lambda_t\le\nu_t$ and
$\lambda_t\le\widetilde\nu_t$, where for measures $\alpha,\beta$ we write
$\alpha\le\beta$ if $\alpha(C)\le\beta(C)$ for every measurable set $C$.

It follows from \cthm{coupling} that there exists a process
$(\zeta_t,\eta_t)$ on $Y\times G$ such that $\zeta_t$ and $\eta_t$ are
ASEP and F-ASEP processes, respectively, $\zeta_0$ is distributed
according to $\muhat$, $\eta_0=\phi(\zeta_0)$, and $\eta_t$ a translate
of $\phi(\zeta_t)$ for all $t$.  Let $\kappa_t$ be the measure on
$Y\times G$ giving the distribution of $(\zeta_t,\eta_t)$; then
$\pi_{Y*}\kappa_t=\muhat\hat S(t)$ and
$\pi_{G*}\kappa_t=(\phi_*\muhat)S(t)$, where $\pi_Y$ and $\pi_G$ are the
projections of $Y\times G$ onto its first and second components,
respectively. Let $m\in\bbz$ be such that $\kappa_t(B_m)>0$, where
$B_m=\{(\zeta,\eta)\in Y\times G\mid\eta=\tau^m\phi(\zeta)\}$, and let
$\widetilde\lambda_t=\rho\phi_*\pi_{Y*}(\ind_{B_m}\kappa_t)$ and
$\lambda_t=\rho\pi_{G*}(\ind_{B_m}\kappa_t$), with $\ind_{B_m}$ the
characteristic function of $B_m$.  Then we have
 \be\label{lamtil}
\widetilde\lambda_t \le \rho\phi_*\pi_{Y*}\kappa_t
  =\rho\phi_*(\muhat\hat S(t))\le\Phi(\muhat\hat S(t))=\widetilde\nu_t,
 \ee
 where we have used \eqref{defPhi}, and 
 \be
\lambda_t\le \rho\pi_{G*}\kappa_t=\rho(\phi_*\muhat)S(t)
   \le\Phi(\muhat)S(t)=\nu_t.
 \ee
 But by the definition of $B_m$, \eqref{lamtil}, and the translation
invariance of $\widetilde\nu_t$,
 \be
\lambda_t=\tau_*^m\widetilde\lambda_t
       \le\tau_*^m\widetilde\nu_t=\widetilde\nu_t.
 \ee
 Since $\lambda_t$ is clearly not zero, by the choice of $m$, the result
follows.\end{proof}

Now we can prove our main result:

\begin{proof}[Proof of \cthm{gibbs}]Let $\rhohat=(2\rho-1)/\rho$; it is
known \cite[Theorem~VIII.3.9(a)]{Liggett} that the Bernoulli measure
$\mu^{(\rhohat)}$ is the unique TIS for the ASEP with density $\rhohat$.
\cthm{TIS}(b) then implies that $\Phi(\mu^{(\rhohat)})$ is the unique TIS
measure for the F-TASEP with density $\rho=1/(2-\rhohat)$.  \end{proof}


\begin{remark}\label{nonTI}(a) Explicit formulas for the measure of
\cthm{gibbs} may be easily obtained from \eqref{defPhi}.  For example,
for $\theta\in\{0,1\}^{\{1,\ldots,m\}}$ with $\theta_i+\theta_{i+1}\ge1$
for $1\le i\le m-1$, we have that (again with $\rhohat=(2\rho-1)/\rho$)
 \be
\begin{split}\Phi(\mu^{(\rhohat)})(\{\eta\mid\eta\rng1m=\theta\})&\\
  &\hskip-70pt=(1-\rho)\left(\frac{1-\rho}\rho\right)^{m-1-\sum_i\theta_i}
   \left(\frac{2\rho-1}\rho\right)^{2\sum_i\theta_i+1-m-\theta_1-\theta_m}.
  \end{split}
 \ee
 This formula was previously obtained in \cite{BESS,BES} in the context of the
 symmetric ($p=1/2)$ facilitated process.

 \smallskip\noindent
 (b) The ASEP system also has non-TI stationary states as long as
$p\ne1/2$, that is, as long as there is a true asymmetry \cite[Example
VIII.2.8]{Liggett}.  We conjecture that this is also true for the F-ASEP,
but we do not have a proof.  \end{remark}

 \medskip\noindent
  {\bf Acknowledgments:} The work of JLL was supported by the AFOSR under
award number FA9500-16-1-0037 and Chief Scientist Laboratory Research
Initiative Request \#99DE01COR.  AA was partially supported by Department
of Science and Technology grant EMR/2016/006624 and by the UGC Centre for
Advanced Studies. We thank two anonymous referees for helpful comments.

 \appendix
\section{An equivalence result\label{same}}

The lemma proved in this appendix is essentially a completion of the
proof of \cthm{coupling}, and we will adopt the notation of that proof;
in particular, we let $\zeta_0$ be the initial ASEP configuration and
write $K:=K_{\zeta_0}$.

\begin{remark}\label{sets}It will be convenient to use a representation
of the sample points for the SAPP and PAPP which is different from that
of \eqref{OmegaP}--\eqref{SAPP} and \eqref{OmegaPo}.  If
$\omega\in\Omega_0$ then, from \eqref{SAPP},
$\omega=\bigl((\omega^{(i,r)},\omega^{(i,l)})\bigr)_{i\in\bbz}$, with
$\omega^{(i,\#)}$ a sequence $\bigl((i,t_j^{(i,\#)})\bigr)_{j\in\bbn}$,
and we may identify $\omega$ with a set of labeled (by $r$ or $l$) {\it
Poisson points} in $\bbz\times\bbr_+\times\{r,l\}$:
$\omega\sim\{(i,t_j^{(i,\#)},\#)\}$.  For given $\omega$ we may and do
assume that the Poisson times $t_j^{(i,\#)}$ are all different and all
nonintegral, since this is true $\Pp$-a.s.~(this avoids discussion of
irrelevant special cases).  A similar representation
$\omegao\sim\{(k,\t_j^{(k,\#)},\#)\}$, with $k\in K$, holds for
$\omegao\in\Omegao_0$.  We say that the PAPP point $(k,\t,\#)$ is
{\it located at} $(i,\t)$ if particle $k$ is at site $i$ at time $\t-$.
\end{remark}

\begin{lemma}\label{equiv} The SAPP and PAPP definitions of the ASEP and
F-ASEP processes are equivalent.  \end{lemma}

\begin{proof} The proofs for the two processes are completely parallel;
to be definite we will consider the ASEP.  We let $\zeta_t$ and
$\zetao_t$ be respectively the SAPP and PAPP ASEP processes, each with
initial configuration $\zeta_0$.  The probability spaces for these
processes, $(\Omega_0,\Pp)$ and $(\Omegao_0,\Po)$, are defined in
\eqref{OmegaP}, \eqref{SAPP}, and \eqref{OmegaPo}; see also \crem{sets}.
We define a map $\Psi:\Omega_0\to\Omegao_0$ as follows: with the point
$\omega\in\Omega_0$ is associated a well-defined space-time history
$\zeta(\omega)=\bigl(\zeta_t(\omega,i)\bigr)_{(i,t)\in\bbz\times\bbr_+}$,
and hence, for each particle $k\in K$, a well defined space-time
trajectory.  We define $\Psi(\omega)$ by the condition that $(k,t,\#)$ is
a PAPP point of $\Psi(\omega)$ iff $(i,t,\#)$ is a SAPP point of $\omega$
which lies on the closure of this trajectory.  Then clearly
$\zeta_t=\zetao_t\circ\Psi$ for all $t$, and the lemma will follow once
we show that $\Psi$ has the correct distribution, that is, that
$\Psi_*\Pp=\Po$.

 We verify this by defining, for each $N\in\bbn$, a certain approximation
$\PsiN$ of $\Psi$.  As a preliminary, let $\JN=[\jnn_1,\jnn_2]$ denote
the interval defined in parallel with the construction of $\IN$ in the
proof of \cthm{coupling}, but using the SAPP rather than the PAPP and
ensuring that $\JN\supset[-N,N]$: $\JN$ is the minimal interval for which
$\jnn_1,\jnn_2\in K$, $\jnn_1\le-N<N\le\jnn_2$, and no Poisson events
occur in the SAPP process for sites $\jnn_1$ and $\jnn_2$ during the time
interval $[0,N]$.  Let $\AN$ be the set of SAPP points $(i,t,\#)$ of
$\omega$ with $(i,t)\in\JN\times[0,N]$.  $\AN$ is a.s.~finite; we let
$\mN$ denote the number of points in $\AN$, and index these points as
$(i_n,t_n,\#_n)_{n=1,\ldots,\mN}$ with $0<t_1<\cdots<t_{\mN}<N$.  By
convention we take $t_0=0$.

We now construct recursively a sequence
$\bigl(\psiNn\bigr)_{n=0,1,\ldots}$ of maps
$\psiNn:\Omega_0\to\Omegao_0$; $\psiNn(\omega)$ will be independent of
$n$ for $n\ge\mN(\omega)$ and $\PsiN$ will then be defined by
$\PsiN:=\psiNN$.  We first take $\psi^{(N,0)}(\omega)$ to be such that,
for each particle $k\in K$, $(k,\t,\#)$ is a PAPP point of
$\psi^{(N,0)}(\omega)$ if and only if it is a SAPP point of $\omega$.
Suppose then that we have defined $\psi^{(N,n-1)}(\omega)$.  To define
$\psiNn(\omega)$ we suppose first that $n\le\mN(\omega)$, consider the
SAPP point $(i_n,t_n,\#_n)$, and let $i_n'$ denote the target site to
which a particle at site $i_n$ {\it might} jump at time $t_n$:
$i_n'=i_n+1$ if $\#_n=r$ and $i_n'=i_n-1$ if $\#_n=l$.  If (in the SAPP
process) either there is no particle at site $i_n$ at time $t_n-$, or the
target site $i_n'$ is occupied at time $t_n-$, then we define
$\psiNn(\omega)=\psi^{(N,n-1)}(\omega)$.  Otherwise, the particle in the
SAPP process at site $i_n$ at $t_n-$, say particle $k$, jumps to site
$i_n'$ in the SAPP process, and $\psiNn(\omega)$ is defined to have the
same Poisson points as $\psi^{(N,n-1)}(\omega)$, {\it except} that we
replace the (labeled) times of the PAPP points for particle $k$ which lie
in the future of $t_n$ with the times of the SAPP points for site $i_n'$
which lie in the future of $t_n$: for $\t>t_n$, $(k,\t,\#)$ is a PAPP
point of $\psiNn(\omega)$ if and only if $(i_n',\t,\#)$ is a SAPP point
of $\omega$.  Continuing in this way we define
$\psi^{(N,0)}(\omega),\ldots,\psiNN(\omega)=:\PsiN(\omega)$.  Finally, if
$n\ge\mN(\omega)$ we take $\psiNn(\omega)=\psiNN(\omega)$.

 We next show that the $\psiNn$ satisfy \rf1--\rf3 below:

\begin{enumerate}[(P1)]

\item\label{prop1} For $\omega\in\Omega_0$, $N\in\bbn$, and
$0\le n\le\mN$, the locations of the set of PAPP points $(k,\t,\#)$ of
$\psiNn(\omega)$ with $(k,\t)\in\JN\times[0,t_n]$ coincide with the PAPP
points of $\Psi(\omega)$ satisfying the same restrictions.

\item\label{prop2} For $\omega\in\Omega_0$, $N\in\bbn$, $0\le n\le\mN$,
and particle $k\in\JN$, if $k$ is located at site $i$ at time $t_n$ then
for $\#=l,r$ the locations of the set of PAPP points $(k,\t,\#)$ of
$\psiNn(\omega)$ with $\t>t_n$ coincide with the set of
SAPP points $(i,t,\#)$ of $\omega$ with $t>t_n$.

\item\label{prop3} $\psiNn_*\Pp=\Po$ for all $n$.

\end{enumerate}

\noindent \rf1--\rf3 are trivially satisfied for $n=0$; to verify them
for general $n$ we argue recursively.  

First, \rf1 for $\psi^{(N,n-1)}$ implies that \rf1 holds for $\psiNn$,
except possibly for PAPP points in $\JN\times(t_{n-1},t_n]$.  Since there
are no SAPP points for $\omega$ in $\JN\times(t_{n-1},t_n)$, and hence no
PAPP points for either $\psiNn(\omega)$ or $\Psi(\omega)$ located in this
region, it remains to show that either (i)~no PAPP point for either
$\psiNn(\omega)$ or $\Psi(\omega)$ is located at $(i_n,t_n,\#_n)$, or
(ii)~a PAPP point $(k,t_n,\#_n)$ for both is located there.  It is clear
that (i) holds if no particle is located at $(i_n,t_n-)$.  On the other
hand, if particle $k$ is located at $(i_n,t_n-)$, then certainly
$(k,t_n,\#_n)$ is a PAPP point of $\Psi(\omega)$; moreover, $k$ must also
be located at $(i_n,t_{n-1})$, so that $(k,t_n,\#_n)$ is a PAPP point of
$\psi{(N,n-1)}(\omega)$ from \rf2 for $\psi^{(N,n-1)}(\omega)$, and so
also of $\psiNn(\omega)$, since a jump at time $t_n$ only changes the
PAPP points in the future of $t_n$.

Second, \rf2 for $\psiNn(\omega)$ follows from \rf2 for
$\psi^{(N,n-1)}(\omega)$ and the observation that if particle $k$ jumps
at time $t_n$ then the change in the PAPP points of this particle which
takes place in passing from $\psi^{(N,n-1)}(\omega)$ to $\psiNn(\omega)$
is precisely what is needed to maintain \rf2.

 Finally, we verify \rf3 for $\psiNn$, assuming \rf3 for
$\psi^{(N,n-1)}$.  We will consider conditional measures
$\Pp(\cdot\mid\Q)$, where $\Q$ is specified by certain events and/or
values of certain random quantities, and the family of all such $\Q$'s
forms a partition of $\Omega_0$.  The $\Q$'s which we use will be
specified during the course of the proof.  For each $\Q$ which arises we
will show that
 \be\label{onQ}
 \psiNn_*\Pp(\cdot\mid\Q)=\psi^{(N,n-1)}_*\Pp(\cdot\mid\Q).
 \ee
 Integrating \eqref{onQ} against the marginal $\Pp(d\Q)$ yields
$\psiNn_*\Pp=\psi^{(N,n-1)}_*\Pp$, which with \rf3 for $\psi^{(N,n-1)}$
yields \rf3 for $\psiNn$.  As a first step, let $\Q_0$ be the event that
$\mN\ge n$. On the complimentary event $\Q_0^c$,
$\psiNn=\psi^{(N,n-1)}\bigl(=\psiNN\bigr)$, so that
\eqref{onQ} holds trivially with $\Q=\Q_0^c$.

Next, we let $\Q_1$ be defined by specifying, in addition to $\Q_0$,
values of the interval $\JN$, of the time $t_n$ (which is well-defined on
$\Q_0$), and of the entire past of $t_n$, including in particular the
values of all $(i_{n'},t_{n'},\#_{n'})$ with $n'\le n$.  For $i\in\bbz$
let $S_i$ be the set of Poisson points $(i,t,\#)$ at site $i$ in the
future of $t_n$.  We can describe the joint distribution of the sets
$S_i$ under $\Pp(\cdot\mid\Q_1)$ in terms of the measure $\kappa_u^{(i)}$
defined, for $u\ge0$, to be the translate by $u$ of the measure
$\lambda^{(i,r)}\times\lambda^{(i,l)}$ (see \eqref{OmegaP}): (i)~the
$S_i$, $i\in\bbz$, are independent; (ii)~$S_i$ is distributed as
$\kappa_{t_n}^{(i)}$ if either (ii.a)~$i\notin\JN$, (ii.b)~$i\in[-N,N]$,
(ii.c)~$i\notin K$, or (ii.d)~$i=i_{n'}$ for some
$(i_{n'},t_{n'},\#_{n'})$ with $n'\le n$; (iii)~$S_i$ has no points in
$(t_n,N]$ and on $(N,\infty)$ is distributed as $\kappa_{N}^{(i)}$, if
$i=\jnn_1$ or $i=\jnn_2$; (iv)~$S_i$ is distributed as the conditional
distribution of $\kappa_{t_n}^{(i)}$, given that there is at least one
point in $(t_n,N)$, otherwise.

Now conditioning on $\Q_1$ determines whether or not a jump takes place
at time $t_n$; let $\Q_1'$ and $\Q_1''$ be $\Q_1$ with the additional
restriction that the jump respectively does or does not take place.
Under $\Q_1''$, $\psiNn=\psi^{(N,n-1)}$, so that \eqref{onQ} holds with
$\Q=\Q_1''$.  On the other hand, under $\Q_1'$, some particle $k$ will
jump from site $i_n$ to $i_n'$; let $\Q_2$ be obtained by specifying
$\Q_1'$ together with values of all the sets $S_i$ for $i\ne i_n,i_n'$.
Consider then \eqref{onQ} with $\Q=\Q_2$; the left side of this equation
is obtained from the right by the replacement of the (labeled) times of
the PAPP points for particle $k$ which lie in the future of $t_n$---and,
by \rf2, these are just the times of $S_{i_n}$---with the times of
$S_{i_n'}$ lying in that same future.  But $S_{i_n}$ and $S_{i_n'}$ have
distributions $\kappa_{t_n}^{(i_n)}$ and $\kappa_{t_n}^{(i_n')}$ under
$\Pp(\cdot\mid\Q_1)$ and hence, by the independence noted in (i) above,
under $\Pp(\cdot\mid\Q_2)$; this is because $i_n$ falls under case
(ii.d), and $i_n'$ under either case (ii.c) or case (ii.d), of the
previous paragraph.  Since $\kappa_{t_n}^{(i_n)}$ and
$\kappa_{t_n}^{(i_n')}$ agree, this verifies \eqref{onQ} for $\Q=\Q_2$
and completes the verification of \rf1--\rf3 for $\psiNn$.

 To complete the proof of the lemma, observe that \rf1, together with the
fact that there are no SAPP points of $\omega$ in $\JN\times(t_{\mN},N]$
and hence no PAPP points of either $\PsiN(\omega)=\psiNN(\omega)$ or
$\Psi(\omega)$ located there, implies that the the set of PAPP points
$(k,\t,\#)$ of $\PsiN(\omega)$ which satisfy $-N\le k\le N$ and
$0\le\t\le N$ coincides with the corresponding set of PAPP points of
$\Psi(\omega)$.  By \rf3, then, the marginal distribution of PAPP points
of $\Psi$ in this region is distributed as the marginal of $\Po$.
Since $N$ is arbitrary, we can conclude that $\Psi_*\Pp=\Po$.
\end{proof}


\begin{thebibliography}{99} \bibitem{bbcs} J. Baik, G. Barraquand, I.
Corwin, and T. Suidan,  Facilitated Exclusion Process. {\it Computation
and Combinatorics in Dynamics, Stochastics and Control}, 1--35, Abel
Symp. {\bf 13}, Springer, Cham, 2018.

\bibitem{BM} Urna Basu and P. K. Mohanty, Active-Absorbing-State Phase
  Transition Beyond Directed Percolation: A Class of Exactly Solvable
  Models.  {\it Phys. Rev. E} {\bf 79}, 041143 (2009).

\bibitem{BESS} Oriane Blondel, Cl\'ement Erignoux, Makiko Sasada, and
Marielle Simon, Hydrodynamic Limit for a Facilitated Exclusion Process.
{\it Annales de l'Institut Henri Poincar\'e, Probabilit\'es et
Statistiques} {\bf 56}, 667714 (2020).

\bibitem{BES} Oriane Blondel, Cl\'ement Erignoux, and Marielle Simon,
Stefan Problem for a Non-Ergodic Facilitated Exclusion Process.  {\it
Probability and Mathematical Physics}
{\bf 2}, 127--178 (2021).

\bibitem{CZ} Dayne Chen and Linjie Zhao, The Limiting Behavior of the
  FTASEP with Product Bernoulli Initial Distribution. arXiv:1801.10612v1
  [math PR].

\bibitem{Oliveira}M\'ario J. de Oliveira,  Conserved Lattice Gas Model with
  Infinitely Many Absorbing States in One Dimension.  {\it Phys. Rev. E} 
  {\bf 71}, 016112 (2005).

\bibitem{gkr}Alan Gabel, P. L. Krapivsky, and S. Redner, Facilitated
  Asymmetric Exclusion.  {\it Phys. Rev. Lett.}  {\bf 105}, 210603
  (2010).

\bibitem{GR} A. Gabel and S. Redner, Cooperativity-Driven Singularities
in Asymmetric Exclusion, {\it J. Stat. Mech.} {\bf 2011}, P06008 (2011).

\bibitem{Georgii} Hans-Otto Georgii, {\it Canonical Gibbs Measures,}
Lecture Notes in Mathematics 760. Springer, Berlin, 1979.

\bibitem{GLS1} S. Goldstein, J. L. Lebowitz, and E. R. Speer, Exact
Solution of the F-TASEP . J. Stat. Mech. 123202 (2019).

\bibitem{GLS2} S. Goldstein, J. L. Lebowitz, and E. R. Speer, The
Discrete-Time Facilitated Totally Asymmetric Simple Exclusion Process.
{\it Pure Appl. Funct. Anal.} {\bf 6}, 177203 (2021).

\bibitem{GLS3} S. Goldstein, J. L. Lebowitz and E. R. Speer, Stationary
States of the One-Dimensional Discrete-Time Facilitated Symmetric
Exclusion Process.  In preparation.

\bibitem{hl}Daniel Hexner and Dov Levine,  Hyperuniformity of
Critical Absorbing States. {\it Phys. Rev. Lett.} {\bf 114}, 110602
(2015).

\bibitem{Hochman} Michael Hochman, {\it Notes on Ergodic Theory}.  {\tt
math.huji.ac.il/ courses/ergodic-theory-2012/notes.final.pdf}.

\bibitem{ISS} Takashi Imamura, Tomohiro Sasamoto, and Herbert Spohn, KPZ,
ASEP and Delta-Bose Gas. {\it J. Phys.: Conf. Ser.} {\bf297} 012016
(2011).

\bibitem{Kallenberg} Olav Kallenberg, {\it Foundations of Modern
Probability, Second Edition.} Springer, New York, 2002.


\bibitem{Liggett}Thomas M. Liggett, {\it Interacting Particle Systems.}
Springer-Verlag, New York, 1985.

\bibitem{mcl}Stefano Martiniani, Paul M. Chaikin, and Dov Levine,
Quantifying Hidden Order out of Equilibrium. {\it Phys. Rev. X} {\bf 9},
011031 (2019).

\bibitem{rpv}Michela Rossi, Romualdo Pastor-Satorras, and Alessandro
Vespignani, Universality Class of Absorbing Phase Transitions with a
Conserved Field. {\it Phys. Rev.  Lett.} {\bf85}, 1803 (2000).

\bibitem{ST}Vladas Sidoravicius and Augusto Teixeira, Absorbing-state
Transition for Stochastic Sandpiles and Activated Random Walks. {\it
Electron.  J. Probab.} {\bf 22}, no. 33, 1–35 (2017).

\bibitem{Spohn} Herbert Spohn, {\it Large Scale Dynamics of Interacting
Particles.}  Springer-Verlag, Berlin, 1991.

\bibitem{Stanley} Stanley, Richard P. {\it Catalan Numbers.} Cambridge
University Press, Cambridge, 2015.

\bibitem{Swart} J. M. Swart, {\it A Course in Interacting Particle
Systems.} arXiv:1703.10007v2 (2020).

\bibitem{Van} P. Vanheuverzwijn, A Note on the Stochastic Lattice Gas
  Model. {\it J. Phys.} {\bf A14}, 1149-1158 (1981). 

\end{thebibliography}
\end{document}